\documentclass[10pt,a4paper]{article}
\usepackage[utf8]{inputenc}
\usepackage[english]{babel}
\usepackage{amsmath}
\usepackage{amsfonts}
\usepackage{amssymb}
\usepackage[left=3cm,right=3cm,top=3cm,bottom=3cm]{geometry} 

\usepackage{amssymb}
\usepackage{amsthm}

\usepackage{amsfonts}
\usepackage{amssymb}
\usepackage{indentfirst}
\usepackage{color,colortbl}
\usepackage[table]{xcolor}
\usepackage{stmaryrd}

\usepackage{array}

\usepackage{float}

\usepackage[pdftex]{graphicx}
\usepackage{epstopdf}
\usepackage{caption}

\usepackage{mathrsfs}
\usepackage{tikz}
\newtheorem{theorem}{Theorem}[section]
\newtheorem{lem}[theorem]{Lemma}
\newtheorem{prop}[theorem]{Proposition}
\newtheorem{corollary}[theorem]{Corollary}
\newenvironment{definition}[1][Definition]{\begin{trivlist}
\item[\hskip \labelsep {\bfseries #1}]}{\end{trivlist}}

\usepackage[titletoc]{appendix} 

\usepackage{hyperref} 
\hypersetup{
colorlinks=true, 
breaklinks=true, 
urlcolor= blue, 
linkcolor= blue, 
citecolor=blue, 
} 

\title{Oracle inequalities for a Group Lasso procedure applied to generalized linear models in high dimension}

\author{M\'elanie~Blaz\`ere,
        Jean-Michel~Loubes
        and~Fabrice~Gamboa}

\date{} 

\begin{document}

\maketitle

\begin{abstract}
We present a Group Lasso procedure for generalized linear  models (GLMs) and we study the properties of this estimator applied to sparse high-dimensional GLMs. Under general conditions on the covariates and on the joint distribution of the pair covariates, we provide oracle inequalities promoting group sparsity of the covariables. We get convergence rates for the prediction and estimation error and we show the ability of this estimator to recover good sparse approximation of the true model. Then we extend this procedure to the case of an Elastic net penalty. At last we apply these results to the so-called  Poisson regression model (the output is modeled as a Poisson process whose intensity relies on a linear combination of the covariables). The Group Lasso method enables to select few groups of meaningful variables among the set of inputs.

\end{abstract}


\section{\textbf{Introduction}}
Handling high dimensional data is nowadays required for  many practical applications ranging from astronomy, economics, industrial problems to biomedicine. Being able to extract information from these large data sets has been at the heart of statistical studies over the last decades and many papers have  extensively studied this setting in a lot of fields ranging from statistical inference to machine learning.  We refer for instance to references therein. 

For high-dimensional data, classical methods based on a direct minimization of the empirical risk can lead to over fitting. Actually adding a complexity penalty enables to avoid it by selecting fewer coefficients. Using an $\ell_0$ penalty leads to sparse solutions but the usely non-convex minimization problem turns out to be extremely difficult to handle when the number of parameters becomes large. Hence the $\ell_1$ type penalty has been introduced to overcome this issue. On the one hand this penalty achieves sparsity of an estimated parameter vector and on the other hand it requires only convex optimization type calculations which are computationally feasible even for high dimensional data.  The use of a $\ell_1$ type penalty, first proposed in \cite{TIB96} by Tibshirani, is now a well established procedure which has been studied in a large variety of models. We refer for example to \cite{BIC09}, \cite{BUN08}, \cite{VAN08}, \cite{FRIED10}, \cite{TAR06}, \cite{ZHANG08} and \cite{BUH11}.

Group sparsity can be promoted by imposing a $\ell_{2}$ penalty to individual groups of variables and then a $\ell_{1}$ penalty to the resulting block norms. Yuan and Lin \cite{YUAN07} proposed an extension of the Lasso in the case of linear regression and presented an algorithm when the model matrices in each group are orthonormal. This extension, called the Group Lasso, encourages blocks sparsity. Wei and Huang \cite{WEI10} studied  the properties of the Group Lasso for linear regression, Nardi and Rinaldo \cite{RIN08} established asymptotic properties  and Lounici, Pontil, van de Geer and Tsybakov \cite{LOU11} stated oracle inequalities in linear Gaussian noise under group sparsity. Meir, van de Geer and B{\"u}hlmann \cite{MEIR08} considered the Group Lasso in the case of logistic regression and Zhang and Huang \cite{ZHANG08} studied the benefit of group sparsity. Another important reference is the work of Negahban, Ravikumar, Wainwright and Yu \cite{NEG12}. In this last paper a unified framework for the study of rates of convergence in high dimensional setting is provided under two key assumptions (restricted strong convexity and decomposability). This work and these two assumptions will be discussed in more details in Section \ref{s:results}.

In this paper we focus on the Group Lasso penalty to select and estimate parameters in the generalized linear model. One of the application is the Poisson regression model. More precisely, we consider the  generalized linear model introduced by McCullagh and Nelder \cite{McC86}. Let $F$ be a distribution on $\mathbb{R}$ and let $(X,Y)$ be a pair of random variables with $X\in \mathbb{R}^{p}$ and $Y \in \mathbb{R}$. The conditional law of  $Y\vert X=x$ is modeled by a distribution from the exponential family and the canonical parameter is the linear predictor. Thus the conditional distribution of the observations given $X=x$
is $P(Y\vert\beta^{*}, x)=\exp(y{\beta^*}^{T}x-\psi({\beta^{*}}^{T}x))$, where ${\beta^{*}}^{T}x$ satisfies $\int\exp(y{\beta^{*}}^{T} x)F(dy)<\infty $ and $\psi$ is the normalized function. Notice that $\mathbb{E}(Y\vert X)\overset{\textrm{a.s}}{=}\psi^{'}({\beta^{*}}^{T}X)$ in other words ${\beta^{*}}^{T}X\overset{\textrm{a.s}}{=}h(\mathbb{E}(Y\vert X))$ where $h={\psi^{'}}^{-1}$ is the so-called link function. Some common examples of generalized linear models are the Poisson regression for count data, logistic and probit regression for binary data or multinomial regression for categorical data. The quantities of interest that we would like to estimate are the component $(\beta^*_{j})_{1\leqslant j\leqslant p}$ of $\beta^*$ and for a given $x$ we may also wish to predict the response $Y\vert X=x$. A natural field of applications for such models is given by genomics and Poisson regression type models. In particular  thousands of variables such as expressions of genes and bacterias can be measured for each animal (mices) in a (pre-)clinical study thanks to the developpement of micro arrays (see, for example, \cite{BRO99} and \cite{DUG99} ). A typical goal is to classify their health status, e.g. healthy or diseased, based on their bio-molecular profile, i.e. the thousands of bio-molecular variables measured for each individual (see for instance \cite{HEL97}, \cite{QUACK06} and \cite{PARK02}).

The paper falls into the following parts. In Section~\ref{s:descrip} we describe the model and the Group Lasso estimator for generalized linear models. In Section~\ref{s:results} we present the main results on coefficients estimation and prediction error and in Section~\ref{s:extension} we consider the model in the particular case of a Poisson regression. We also study here the general model in the case of a mixture of an $\ell_{1}$ and $\ell_{2}$ penalty. The Group Lasso estimator is then used on simulated data sets in Section~\ref{s:simu} and its performances are compared to those of the Lasso estimator. The Appendix is devoted to the proof of the main theorems.

\section{\textbf{Sparse Variable selection for generalized linear models}} \label{s:descrip}

\subsection{\textbf{The model}}

The Exponential family on the real line is a unified family of distributions parametrized by a one dimensional parameter and is widely used for practical modelling. Let $F$ be a probability distribution on  $\mathbb{R}$ not concentrated on a point and 
$$\Theta:= \left\lbrace \theta \in \mathbb{R}:\: \int\exp(\theta x)F(dx)<\infty \right\rbrace.$$ 
Define
$$M(\theta):= \int\exp(\theta x)F(dx) \quad (\theta \in \Theta ) $$  
and 
$$\psi(\theta):= \log(M(\theta))\quad (\theta \in \Theta ).$$ 
Let $P(y;\theta)=\exp(\theta y -\psi(\theta))$ with $\theta \in \Theta$. The densities of probability (related to a mesure adapted to the continuous or discrete case) $\left\lbrace P(.;\theta): \theta \in \Theta \right\rbrace $ is called the exponential family. $\Theta$ is the natural parameter space and $\theta$ is called the canonical parameter. The exponential family includes most of the commonly used distributions like normal, gamma, poisson or binomial distributions \cite{McC86}.

We consider a pair of random variables $(X,Y)$ where $Y \in\mathbb{R}$ and $X\in \mathbb{R}^{p}$ such that the  conditional distribution  $Y\vert X=x$  is
$P(Y\vert\beta^{*},x)=\exp(y{\beta^{*}} ^{T}x-\psi({\beta^{*}} ^{T}x))$,  
with ${\beta^{*}}^{T}x \in \Theta$. 
Our aim is to estimate  the components $(\beta_{j}^{*})_{1\leqslant j\leqslant p}$ of $\beta^{*}$ in order  to predict the response $Y\vert X=x$ conditionally on a given value of $x$. We assume that
 \begin{itemize}
\item [\textbullet] (H.1): the variable $X$ is almost surely bounded by a constant $L$ i.e. there exists a constant $L>0$ such that $\parallel X\parallel_{\infty} \leqslant L$ a.s.
\item [\textbullet](H.2): for all $x \in \left[-L,L \right]^{p} $, ${\beta^{*}}^{T}x \in \textrm{Int}(\Theta)$
\end{itemize} 
and we consider $$\Lambda=\left\lbrace  \beta \in \mathbb{R}^{p}:\: \forall x \in \left[-L,L \right]^{p},\beta^{T}x \in \Theta \right\rbrace. $$
Let $(X_{1},Y_{1}),...,(X_{n},Y_{n})$ be i.i.d copies of $(X,Y)$ where $X_{i}=(X_{i,1},...,X_{i,p})^{T}$ for $i=1,...,n$. We consider the case of high-dimensional regression i.e $p\gg n$. The log-likelihood for a generalized linear model is given by 
$$\mathcal{L}(\beta)=\sum_{i=1}^{n}\left[Y_{i}\beta^{T}X_{i}-\psi(\beta^{T}X_{i}) \right]. $$
We denote the generalized linear model loss function by
\begin{equation}
l(\beta)=:l(\beta;x,y)=:-yf_{\beta}(x)+\psi(f_{\beta}(x)),
\label{loss-function}
\end{equation}
where $f_{\beta}(x):=\beta^{T}x$.
Notice that this function is convex in $\beta$ (as $\psi$ is convex).
The associated risk is defined by $\mathbb{P}l(\beta)=:\mathbb{E}l
(\beta;Y,X)$ and the empirical risk by $\mathbb{P}_{n}l(\beta)$  where
$$\mathbb{P}_{n}l(\beta):=\frac{1}{n}\sum_{i=1}^{n}\left[ -Y_{i}\beta^{T}X_{i}+\psi(\beta^{T}X_{i})\right] . $$
Obviously
$$\beta^{*}=\underset{\beta \in \Lambda}{\textrm{argmin}}\:\mathbb{P}l(\beta).$$

\subsection{\textbf{Group Lasso for generalized linear model}}

Assume that $X$ is structured into $G_{n}$ groups each of size $d_{g}$ for $g \in \left\lbrace 1,...,G_{n} \right\rbrace $. For $i=1,...,n$ we set $$X_{i}=(X^{1}_{i},...,X^{g}_{i},...,X^{G_{n}}_{i})^{T}$$ 
where $$ X^{g}_{i}=(X_{i,1}^{g},...,X_{i,d_{g}}^{g})^{T}$$ and $\sum_{g=1}^{G_{n}}d_{g}=p$. This decomposition is often natural in biology and micro arrays data when the covariates are genes expresssion (see for instance \cite{SEGAL03} and \cite{WEST01}). We allow the number of groups to increase with the sample size $n$, so we can consider the case where $G_{n}\gg n$. Define $d_{\textrm{max}}:= \underset{g \in \left\lbrace 1,...,G_{n}\right\rbrace }{\textrm{max}}d_{g}$ and $d_{\textrm{min}}:= \underset{g \in \left\lbrace 1,...,G_{n}\right\rbrace }{\textrm{min}}d_{g}$. For $\beta \in \mathbb{R}^{p}$ we denote by $\beta^{g}$ the sub-vector of $\beta$ whose indexes correspond to the index set of the $g^{th}$ group of $X$.

Let us consider the Group Lasso estimator which achieves group sparsity and is obtained as the solution of the convex optimization problem
$$
\hat{\beta}_n=\underset{\beta \in \Lambda}{\textrm{argmin}}\left\lbrace \mathbb{P}_{n}l(\beta)+\sum_{g=1}^{G_{n}}s(d_{g})\Vert \beta^{g}\Vert_{2}\right\rbrace 
$$
where $r_{n}$ is the tuning parameter.

Here $\Vert .\Vert_{2}$ refers to the Euclidian norm and $s$ is a given function. An increase in $r_{n}$ leads to a diminution of the $\beta^{g}$ to zero, this means that some blocks become simultaneously zero and groups of predictors drop out of the model. Typically we choose $s(d_{g}):=\sqrt{ d_{g}}$ to penalize more heavily groups of large size. Notice that if all the groups are of size one then we recover the Lasso estimator. The Group Lasso achieves variables selection and estimation simultaneously as the Lasso does. The penalty function is the sum of the $\ell_{2}$ norm of the groups of variables. Thus the Group Lasso estimator acts like the Lasso at the group level \cite{RIN08}, \cite{LOU11}, \cite{YUAN07}. Actually the objective function above is the sum of a particular loss function (which is convex and based on the exponential family) with a weighted regularizer. This type of convex optimization problem is referred as a regularized M-estimator in the paper of Negahban et al. \cite{NEG12}. We can also notice that the penalty norm satisfies the decomposability condition defined in this last paper.

We study estimation and prediction properties of the Group Lasso in high dimensional settings when the number of groups exceeds the sample size i.e. $G_{n}\gg n$. Define $H^{*}=\left\lbrace g: {\beta^{*}}^{g}\neq 0\right\rbrace $ the index set of the groups for which the correponding sub vectors of $\beta^{*}$ are non-zero and $m^{*}:=\vert H^{*}\vert$. Such a set characterizes the sparsity of the model. In the following ${H^{*}}^{c}$ denotes the index set of the groups which are not in $H^{*}$. We can notice that $H^{*}$ and $m^{*}$ depend on $n$ but for simplicity we do not specify this dependency. In general, it will be hopeless to estimate
all unknown parameters from data except if we make the assumption that the true parameter is group sparse. In this paper we consider that $\beta^{*}$ is partitioned into a number of groups, in correspondance with the partition of $X$, only few of which are relevant. The index of group sparsity $m^{*}$ will be discussed more deeply after Theorem \ref{theo-gp-glm-norm2}. We also assume (H.3): there exists a constant $B>0$ such that $\sum_{g=1}^{G_{n}}\sqrt{d_{g}}\Vert {\beta^{*}}^{g}\Vert_{2}\leqslant B$.

\section{\textbf{Main Results}}
\label{s:results}

\subsection{\textbf{Bounds for estimation and prediction error}}

Under some assumptions on the value of the parameter $r_{n}$ and on the covariance matrix of $X$ we are going to show the ability of this estimator to recover good sparse approximation of the true model. To prove oracle inequalities for the Group Lasso applied to generalized linear model we need to state concentration inequalities for the empirical process $\mathbb{P}_{n} \left( l(\beta)\right)$ for $\beta \in \Lambda$. 
This step is essential to compute an appropriate lower bound for the regularization parameter that ensures good statistical properties of the estimator with high probability. Notice that this step requires a boundedness assumption on the $\left( X_{i,j}\right) $ (cf. proof of Proposition \ref{prop:probaAetB}).

To state concentration inequalities we first break down the empirical process into a linear part and a part which depends on the normalized parameter $\psi$ $$\left( \mathbb{P}_{n}-\mathbb{P}\right) \left( l(\beta)\right)=\left( \mathbb{P}_{n}-\mathbb{P}\right) \left( l_{l}(\beta)\right)+\left( \mathbb{P}_{n}-\mathbb{P}\right) \left( l_{\psi}(\beta)\right)$$ where $l_{l}(\beta):=l_{l}(\beta,x,y)=-y\beta^{T}x$ and $l_{\psi}(\beta):=l_{\psi}(\beta,x)=\psi(\beta^{T}x).$  Define 
$$\mathcal{A}=\bigcap_{g=1}^{G_{n}}\left\lbrace L_{g}\leqslant r_{n}/2\right\rbrace $$
where $$L_{g}:=\left\Vert \frac{1}{\sqrt{d_{g}}n}\sum_{i=1}^{n}\left( Y_{i}X_{i}^{g}-\mathbb{E}(YX^{g})\right) \right\Vert_{2}$$ for all $g\in \left\lbrace 1,...,G_{n} \right\rbrace $ and
$$\mathcal{B}= \left\lbrace \underset{\beta: \sum_{g=1}^{G_{n}}\sqrt{d_{g}}\Vert \beta^{g}-{\beta^{*}}^g\Vert_{2}\leqslant M}{\textrm{sup}}\left\lvert \nu_{n}(\beta, \beta^{*}) \right\rvert \leqslant \dfrac{r_{n}}{2}\right\rbrace$$ where $$\nu_{n}(\beta, \beta^{*}):=\dfrac{(\mathbb{P}_{n}-\mathbb{P})\left( l_{\psi}(\beta^{*})-l_{\psi}(\beta)\right)}{\sum_{g=1}^{G_{_{n}}}\sqrt{d_{g}}\Vert \beta^{g}-{\beta^{*}}^g\Vert_{2}+\varepsilon_{n}}$$ with $M=8B+\varepsilon_{n}$ and $\varepsilon_{n}=\frac{1}{n}$. We assume that $G_{n}$ and $n$ are such that $\frac{\log(2G_{n})}{n}\leqslant 1$. The next proposition shows that the event $\mathcal{A}\cap \mathcal{B}$ occurs with high probability for some suitable values of the tuning parameter.
\begin{prop}\label{prop:probaAetB}
Let $$r_{n}\geqslant AKL \left\lbrace  C_{L,B}\vee\underset{\left\lbrace\vert x\vert\leqslant L\kappa_{n}\right\rbrace\cap \Theta}{\textrm{max}}\vert \psi^{'}(x)\vert)\right\rbrace \sqrt{\dfrac{2\log(2G_{n})}{n}}$$ where $K$ is a universal constant, $A>\sqrt{2}$, $\kappa_{n}:=17B+\frac{2}{n}$ and $C_{L,B}$ is defined in Lemma \ref{lem-moment}. We have 
$$\mathbb{P}\left(\mathcal{A}\cap \mathcal{B} \right)\geqslant 1-(C+2d_{\textrm{max}})(2G_{n})^{-A^{2}/2}.$$ where $C$ is a universal constant.
\end{prop}

The proof rests on concentration inequalities and is detailled in the appendix. Indeed infering a bound for the probability of the events $\mathcal{A}$ and $\mathcal{B}$ is equivalent to prove concentration inequalities for the linear and non linear part of the empirical process. A concentration inequality for the linear part is derived from Bernstein inequality once the following lemma has been proved. This lemma provides moment bounds for $Y$.
\begin{lem}\label{lem-moment}
Let $(X,Y)$ a pair of random variables whose conditional distribution is
$P(Y;\beta^{*}\vert X=x)=\exp(y{\beta^{*}} ^{T}x-\psi({\beta^{*}} ^{T}x))$ and assume assumptions (H.1-3) are fulfilled. For all $k\in \mathbb{N}^{*}$ there exists a constant $C_{L,B}$ (which depends only on $L$ and $B$) such that $\mathbb{E}(\vert Y\vert^{k})\leqslant k!(C_{L,B})^{k}$.
\end{lem}
The boundedness assumption on the components of $X$ is required to prove this lemma. Then, for the non linear part of the empirical process, we use again this assumption to show that we can restrict the study of $\psi$ to a suitable compact set.
Since $\psi$ is lipchitzian on this compact set, concentration results for lipchitzian loss functions (see~\cite{LEDOUX91}) allow to bound the probability of the event $\mathcal{B}$.

Thus, on the event $\mathcal{A}$ which occurs with high probability (see Proposition \ref{prop:probaAetB}), we have an upper bound for the linear part of the empirical process $(\mathbb{P}_{n}-\mathbb{P})\left( l_{l}(\beta^{*})-l_{l}(\hat{\beta}_n)\right)$.
\begin{prop}
\label{prop:upbound1}
On the event $\mathcal{A}$ 
$$(\mathbb{P}_{n}-\mathbb{P})\left( l_{l}(\beta^{*})-l_{l}(\hat{\beta}_n)\right)\leqslant \frac{r_{n}}{2}\sum_{g=1}^{G_{n}}\sqrt{d_{g}}\Vert \hat{\beta}_{n}^{g}-{\beta^{*}}^g\Vert_{2}.$$
\end{prop}
\begin{proof}
We have
$$(\mathbb{P}_{n}-\mathbb{P})\left( l_{l}(\beta^{*})-l_{l}(\hat{\beta}_n)\right)$$
$$=\sum_{g=1}^{G_{n}}(\hat{\beta}_n^{g}-{\beta^{*}}^{g})^{T}\left[ \frac{1}{n}\sum_{i=1}^{n}Y_{i}X_{i}^{g}-\mathbb{E}(YX^{g})\right]  $$
$$\leqslant\sum_{g=1}^{G_{n}}\sqrt{d_{g}}\Vert \hat{\beta}_n^{g}-{\beta^{*}}^{g}\Vert_{2}\left\Vert \frac{1}{\sqrt{d_{g}}n}\sum_{i=1}^{n}\left( Y_{i}X_{i}^{g}-\mathbb{E}(YX^{g})\right) \right\Vert_{2}.$$
The last bound is obtained by using Cauchy-Schwarz inequality. Then on the event $\mathcal{A}$ the proposition follows.
\end{proof}
Therefore the difference between the linear part of the empirical process and its expectation is bounded from above by the tuning parameter multiplied by the norm (associated to the Group Lasso penalty) of the difference between the estimated parameter and the true parameter $\beta^{*}$.
We can state a similar result for the non linear part of the empirical process, the key of the proof is based on the fact that the estimator $\hat{\beta}_n$ is in a neighborhood of the target parameter $\beta^{*}$ on the event $\mathcal{A}\bigcap\mathcal{B}$.
\begin{lem}\label{local}
On the event $\mathcal{A}\bigcap\mathcal{B}$ we have
$ \sum_{g=1}^{G_{n}}\sqrt{d_{g}}\Vert \hat{\beta}_n^{g}-{\beta^{*}}^g\Vert_{2}\leqslant M,$ where we recall that $M=8B+\varepsilon_{n}$ and $\varepsilon_{n}=\frac{1}{n}$.
\end{lem}
Then the next proposition provides an upper bound for $(\mathbb{P}_{n}-\mathbb{P})\left( l_{\psi}(\beta^{*})-l_{\psi}(\hat{\beta}_n)\right)$ and is directly involved by the definition of $\mathcal{B}$ and Lemma~\ref{local}.
\begin{prop}\label{prop:upbound2}
On the event $\mathcal{A}\cap \mathcal{B}$  
$$(\mathbb{P}_{n}-\mathbb{P})\left( l_{\psi}(\beta^{*})-l_{\psi}(\hat{\beta}_{n})\right)$$
$$\leqslant \dfrac{r_{n}}{2}\left( \sum_{g=1}^{G_{n}}\sqrt{d_{g}}\Vert \hat{\beta}_{n}^{g}-{\beta^{*}}^g\Vert_{2} +\varepsilon_{n}\right).  $$
\end{prop}

Once concentration of the loss function around its mean is stated, we have to ensure that the loss function is not too flat in such a way that if the loss difference $l(\hat{\beta}_{n})-l(\beta^{*})$ converges to zero then $\hat{\beta}_{n}$ converges to $\beta^{*}$. In this paper such a property holds assuming that the covariance matrix satisfies a Group Stabil condition (see condition below). This condition is closely related to the one of Negahban et al. \cite{NEG12} called Restricted Strong Convexity property. Notice that in our analysis the boundedness of the covariates is required to prove such a property. In fact, thanks to the boundedness of the covariates, we first bound from below the mean deviation of the loss function by a quadratic function (see Proposition \ref{prop:lowbound} in Appendix A). This first step enables to relate the deviation in the loss to the deviation of the estimated parameter from the true one. Then the  Group Stabil condition implies that the loss function satisfies a local strong convexity property. This has been characterized as a common step for convergence of $M$-estimator (see \cite{NEG12}).
Notice that the boundedness assumption on the components of $X$ is not required to obtain a kind of strong convexity. For example, as stated by Negahban et al., if the covariates have sub-gaussian tails and if the covariance matrix is positive definite then the loss function satisfies a form of restricted strong convexity property with high probability. However, as noticed above the boundedness assumption is necessary to establish Proposition \ref{prop:probaAetB}. 

Therefore the key condition to derive oracle inequalities rests on the correlation between the covariates i.e. on the behavior of the Gram matrix $\frac{1}{n}\sum_{i=1}^{n}X_{i}X_{i}^{T}$ which is necessarily singular when $p>n$. Meier, van de Geer and B{\"u}hlmann \cite{MEIR08} proved that the group Lasso is consistent in the particular case of logistic regression and gave bounds for the prediction error under the assumption that $\mathbb{E}(XX^{T})$ is non singular. In this paper we give sharp bounds for estimation and prediction errors for generalized linear models using a weaker condition similar to the restricted eigenvalue condition (RE) of Bickel, Ritov and Tsybakov \cite {BIC09}. This condition is quite weaker than the one of Bunea, Tsybakov and Wegkamp \cite{BUN07}. In their article  Bickel, Ritov and Tsybakov also give several sufficient conditions for RE (which are easier to check). Here we use a condition which is a group version of the Stabil Condition first introduced by Bunea \cite {BUN08} for logistic regression in the case of an $\ell_{1}$ penalty. This condition is similar (within a constant $\varepsilon$) to the condition used by Lounici, Pontil, van de Geer and Tsybakov \cite{LOU11} to state oracle inequalities for linear regression.
For $c_{0}$, $\varepsilon>0$, we define the restricted set as
$$
S(c_{0},\varepsilon)=\left\lbrace \delta: \sum_{g \in {H^{*}}^{c} }\sqrt{d_{g}}\Vert \delta^{g}\Vert_{2}\leqslant c_{0}\sum_{g \in H^{*} }\sqrt{d_{g}}\Vert \delta^{g}\Vert_{2}+\varepsilon \right\rbrace. 
$$

On this set we assume that the covariance matrix satisfies the Group Stabil condition defined below. This condition ensures local stong convexity in a neighborhood of $\beta^{*}$.

Let $\Sigma:=\mathbb{E}(XX^{T})$ be the $p \times p$  covariance matrix.
\begin{definition}: {\bf Group Stabil Condition}\\
Let $c_{0},\varepsilon >0$ be given. $\Sigma$ satisfies the Group Stabil condition $GS(c_{0},\varepsilon,k)$ if there exists $0<k<1$ such that 
 $$ \delta^{T} \Sigma \delta\geqslant k\sum_{g \in H^{*} }\Vert \delta^{g}\Vert_{2}^{2}-\varepsilon $$ for any $\delta \in S(c_{0},\varepsilon)$.
\end{definition}

Before going any further we have to define two norms that are used to control the estimation error. For all $z \in \mathbb{R}^p$, let
$$\Vert z\Vert_{R}:=\sum_{g=1}^{G_{n}}\sqrt{d_{g}}\Vert z^g \Vert_{2}$$
and 
$$\Vert z\Vert_{2,1}:=\sum_{g=1}^{G_{n}}\Vert z^g \Vert_{2}.$$
The norm $\Vert .\Vert_{R}$ is called the regularizer norm.

We are now able to state the main result of this paper which provides meaningful bounds for the estimation and prediction error when the true model is sparse and 
$\log(G_{n})$ is small as compared to $n$.

Let $\gamma^{*}:=\sum_{g\in H^{*}}d_{g}$. We recall that $\kappa_{n}:=17B+\frac{2}{n}$.
\begin{theorem}
Assume condition $GS(3,\frac{1}{2n},k)$ is fulfilled. Let $$r_{n}\geqslant AKL \left\lbrace  C_{L,B}\vee\underset{\left\lbrace\vert x\vert\leqslant L\kappa_{n}\right\rbrace\cap \Theta}{\textrm{max}}\vert \psi^{'}(x)\vert)\right\rbrace \sqrt{\dfrac{\log(2G_{n})}{n}}$$ where $A>\sqrt{2}$, $K$ is a universal constant and $C_{L,B}$ is defined in Lemma \ref{lem-moment}. Then, with probability at least $1-(C+2d_{\textrm{max}})(2G_{n})^{-A^{2}/2}$ (where $C$ is given in Proposition \ref{prop:probaAetB}), we have

$$\Vert \hat{\beta}_n-{\beta^{*}} \Vert_{R}\leqslant \dfrac{4}{c_{n}k}r_{n}\gamma^{*}+ \left( 1+\frac{1}{r_{n}}\right) \frac{1}{2n},$$

$$\Vert \hat{\beta}_n-{\beta^{*}} \Vert_{2,1}\leqslant \dfrac{4}{c_{n}k\sqrt{d_{\textrm{min}}}}r_{n}\gamma^{*}+ \left( 1+\frac{1}{r_{n}}\right) \frac{1}{2n\sqrt{d_{\textrm{min}}}}$$
and
$$\mathbb{E}\left( \hat{\beta}_n^{T}X-{\beta^{*}}^{T}X\right) ^{2}\leqslant \dfrac{16}{c_{n}^{2}k}r_{n}^{2}\gamma^{*}+ \dfrac{2r_{n}+1}{2c_{n}n}. $$
where $$c_{n}:=\underset{\left\lbrace\vert x\vert\leqslant L(9B+\frac{1}{n})\right\rbrace\cap\Theta }{\textrm{min}}\left\lbrace  \frac{\psi^{''}(x)}{2}\right\rbrace. $$
\label{theo-gp-glm}
\end{theorem}

Notice that $c_{n}>0$ since the measure associated to the distribution $F$ is not concentrated on a point.  

These results are similar to those of Nardi and Rinaldino \cite{RIN08} who proved asymptotic properties of the Group Lasso estimator for linear models. 
We can notice that if $\gamma^{*}=O(1)$ then the bound on the estimation error is of the order $O\left( \sqrt{ \frac{\log G_{n} }{n}}\right) $ and the Group Lasso estimator still remains consistent for the $\ell_{2,1}$-estimation error and for the $\ell_{2}$-prediction error under the Group Stabil condition if the number of groups increases almost as fast as $O(\exp(n))$. The term  $\sqrt{\log G_{n}}$ is the price to pay for having a large number of factors and not knowing where are the non zero ones. 

Since $\Vert \hat{\beta}_n-{\beta^{*}} \Vert_{2}\leq\Vert \hat{\beta}_n-{\beta^{*}} \Vert_{2,1}$, we also have 
$$\Vert \hat{\beta}_n-{\beta^{*}} \Vert_{2}\leqslant \dfrac{4}{c_{n}k\sqrt{d_{\textrm{min}}}}r_{n}\gamma^{*}+ \left( 1+\frac{1}{r_{n}}\right) \frac{1}{2n\sqrt{d_{\textrm{min}}}}.$$
However, a sharper bound for the $l_{2}$-norm of the estimation error holds but under a stronger assumption than $GS(3,\frac{1}{2n},k)$.
\begin{theorem}\label{theo-gp-glm-norm2}
Assume $GS(2m^{*},3,\frac{1}{2n},k^{'})$ i.e there exists $0<k^{'}<1$ such that 
 $$ \delta^{T}\Sigma \delta\geqslant k^{'}\sum_{g \in J}\Vert \delta^{g}\Vert_{2}^{2}- \frac{1}{2n}$$ for any $\delta $ such that $\sum_{g \in {J}^{c} }\sqrt{d_{g}}\Vert \delta^{g}\Vert_{2}\leqslant 3\sum_{g \in J }\sqrt{d_{g}}\Vert \delta^{g}\Vert_{2}+\frac{1}{2n} $ and $J$ such that $\vert J \vert \leq 2m^{*}$.
Then we have with probability at least $1-(C+2d_{\textrm{max}})(2G_{n})^{-A^{2}/2}$
$$\Vert \hat{\beta}_n-{\beta^{*}} \Vert_{2}^{2}\leqslant 10\dfrac{ d_{\textrm{max}}}{d_{\textrm{min}}}\left\lbrace \dfrac{ 16}{{k^{'}}^{2}c_{n}^{2}}r_{n}^{2}\gamma^{*}+ \frac{2r_{n}+1}{2k^{'}c_{n}n}+\dfrac{1}{2k^{'}n}\right\rbrace .$$
\end{theorem}
\begin{proof}
Let $J^{*}=H^{*} \cup I$ and $I$ is the set of indice corresponding to the $m^{*}$ largest values of $\sqrt{d_{g}}\Vert  \hat{\beta}_n^{g}-{\beta^{*}}^g  \Vert_{2}$ in ${H^{*}}^{c}$.
We can prove (see proof of Theorem 3.1 in \cite{LOU11} with $\lambda_{g}=r_{n}\sqrt{d_{g}}$ ) that
\begin{equation}
\label{eq:norm-l2-GS(2m)-0}
\sum_{g \in {J^{*}}^{c}}\Vert \hat{\beta}_n^{g}-{\beta^{*}}^g\Vert_{2}^{2} \leq 9\dfrac{d_{\textrm{max}}}{d_{\textrm{min}}}\sum_{g \in J^{*}}\Vert \hat{\beta}_n^{g}-{\beta^{*}}^g\Vert_{2}^{2}
\end{equation}
and $$\sum_{g \in {J^{*}}^{c} }\sqrt{d_{g}}\Vert \hat{\beta}_n^{g}-{\beta^{*}}^g\Vert_{2}\leqslant 3\sum_{g \in J^{*} }\sqrt{d_{g}}\Vert \hat{\beta}_n^{g}-{\beta^{*}}^g\Vert_{2}+\frac{1}{2n} .$$ 
Therefore on one hand, using assumption $GS(2m^{*},3,\frac{1}{2n},k^{'})$, we deduce
\begin{equation}\label{eq:norm-l2-GS(2m)-1}
k^{'}\sum_{g \in J^{*}}\Vert \hat{\beta}_n^{g}-{\beta^{*}}^g\Vert_{2}^{2}\leq  \mathbb{E}\left( \hat{\beta}_n^{T}X-{\beta^{*}}^{T}X\right) ^{2}+ \dfrac{1}{2n}
\end{equation}
and on the other hand, by the same arguments as those used in the proof of Theorem \ref{theo-gp-glm} to state Equation (\ref{eq:predict-error}), we have 
 \begin{equation}
\label{eq:norm-l2-GS(2m)-2}
\mathbb{E}\left( \hat{\beta}_n^{T}X-{\beta^{*}}^{T}X\right) ^{2}\leqslant \dfrac{16}{c_{n}^{2}k^{'}}r_{n}^{2}\gamma^{*}+ \dfrac{2r_{n}+1}{2nc_{n}}. 
\end{equation}
From equation (\ref{eq:norm-l2-GS(2m)-1}) and equation (\ref{eq:norm-l2-GS(2m)-2}) we conclude
 \begin{equation}
\label{eq:norm-l2-GS(2m)-3}
\sum_{g \in J^{*}}\Vert \hat{\beta}_n^{g}-{\beta^{*}}^g\Vert_{2}^{2}\leqslant \dfrac{16}{{k^{'}}^{2}c_{n}^{2}}r_{n}^{2}\gamma^{*}+ \dfrac{2r_{n}+1}{2k^{'}c_{n}n}+\dfrac{1}{2k^{'}n}. 
\end{equation}
Finally inequalities (\ref{eq:norm-l2-GS(2m)-0}) and (\ref{eq:norm-l2-GS(2m)-3}) conclude the proof.
\end{proof}

The convergences rates obtained in Theorem \ref{theo-gp-glm} and Theorem \ref{theo-gp-glm-norm2} are exactly of the same order as the ones stated by Lounici and al. \cite{LOU11} for the Group Lasso in a Gaussian setting. The oracle inequality stated in Theorem \ref{theo-gp-glm-norm2} shows that the $l_{2}$-estimation error is bounded by $O\left(\gamma^{*}\frac{\log G_{n}}{n} \right)$ under $GS(2m^{*},3,\frac{1}{2n},k^{'})$. Therefore, in the case of finite size groups, $m^{*}$ still could be much larger than $\log G_{n}$  and the estimator remains consistent. We can also notice that the number of samples required in order that the prediction and estimation error (with respect to the $l_{2}$ norm) goes to zero is almost of the order of $O\left( \gamma^{*}\log G_{n}\right) $. 

As mentionned above the group structured norm satisfies the decomposability property (see \cite{NEG12}). Furthermore, under the assumptions made, the loss function satisfies a local restricted strong convexity property. According to Negahban et al., these properties are two important conditions that ensure good statistical properties of M-estimators. This is especially true for the Group Lasso applied to generalized linear model as shown in Theorem \ref{theo-gp-glm} and Theorem \ref{theo-gp-glm-norm2}. Indeed, these two theorems demonstrate the ability of the Group Lasso to recover good approximation of the true model for sparse generalized linear models under the Group Stabil condition.

\subsection{\textbf{Lasso for generalized linear models}}

When each group is of size one we recover the Lasso estimator 
$$\hat{\beta}_n=\underset{\beta\in \Lambda}{\textrm{argmin}}\left\lbrace  \mathbb{P}_{n}l(\beta)+2r_{n}\Vert \beta\Vert_{1}  \right\rbrace $$
where $\Vert \beta\Vert_{1} =\sum_{j=1}^{n}\vert\beta_{j}\vert$. Thus following step by step the proof of Theorem \ref{theo-gp-glm} we can easily deduce bounds for estimation and prediction error for the Lasso estimator in the case of generalized linear models. Notice that the $l_{2}$-estimation error of the Lasso applied to GLMs was first studied by Negahban and al. (see \cite{NEG09} and \cite{NEG12}). The Lasso is a special case of the Group Lasso where $\gamma^{*}=s^{*}$ with $s^{*}:=\vert I^{*}\vert=\left\lvert\left\lbrace j:\beta_{j}^{*}\neq 0 \right\rbrace \right\rvert$ and $G_{n}=p$. We still consider high-dimensional data i.e. $n\ll p$ and sparsity assumption on the target $\beta^{*}$ i.e. $s^{*}\ll p$ and we assume (H.1-3) except that for (H.3) we consider the $\ell_{1}$ norm i.e. $\Vert \beta^{*}\Vert_{1}\leqslant B$. The condition $GS$ in this case requires the existence of $0<k<1$ such that $\delta ^{T}\Sigma \delta\geqslant k\sum_{j\in I^{*}} \delta_{j}^{2}-\varepsilon$ for any $\delta \in S(c_{0},\varepsilon)=\left\lbrace \delta \in \mathbb{R}^{p}: \sum_{j \in {I^{*}}^{c} }\vert \delta_{j}\vert\leqslant c_{0}\sum_{j \in I^{*} }\vert \delta_{j}\vert+\varepsilon \right\rbrace $. We recover the Stabil condition $St(c_{0},\varepsilon,k)$ of \cite{BUN08}. 
We also define condition $St(2s^{*},c_{0},\varepsilon,k^{'})$ i.e there exists $0<k^{'}<1$ such that $\delta ^{T}\Sigma \delta\geqslant k^{'}\sum_{j\in J} \delta_{j}^{2}-\varepsilon $ for any $\delta $ which satisfies $\sum_{j \in {J}^{c} }\vert \delta_{j}\vert\leqslant c_{0}\sum_{j \in J }\vert \delta_{j}\vert+\varepsilon$ and $J$ such that $\vert J \vert \leq 2s^{*}$.
\begin{theorem}\label{th:LassoGLM}
Assume condition $St(3,\frac{1}{2n},k)$ is fulfilled. Let $$r_{n}\geqslant  AKL\left\lbrace  C_{L,B}\vee\underset{\left\lbrace\vert x\vert\leqslant L\kappa_{n}\right\rbrace\cap\Theta}{\textrm{max}}\vert \psi^{'}(x)\vert\right\rbrace \sqrt{\dfrac{\log(2p)}{n}}$$ where $A>\sqrt{2}$ and $C_{L,B}$ depends only on $L$ and $B$. We have, with probability at least $1-C(2p)^{-A^{2}/2}$ (where $C$ is a universal constant),
$$\Vert \hat{\beta}_n-\beta^{*}\Vert_{1}\leqslant \dfrac{4}{c_{n}k}r_{n} s^{*}+\left( 1+\frac{1}{r_{n}}\right) \dfrac{1}{2n} $$ and
$$ \mathbb{E}\left( \hat{\beta}_n^{T}X-{\beta^{*}}^{T}X\right) ^{2}\leqslant \dfrac{16}{c_{n}^{2}k}r_{n}^{2}s^{*}+\dfrac{2r_{n}+1}{2nc_{n}}. $$
Furthermore if $St(2s^{*},3,\frac{1}{2n},k^{'})$ holds then we have
$$
\Vert \hat{\beta}_n-{\beta^{*}} \Vert_{2}^{2}\leqslant 10\dfrac{ d_{\textrm{max}}}{d_{\textrm{min}}}\left\lbrace \dfrac{16}{{k^{'}}^{2}c_{n}^{2}}r_{n}^{2}s^{*}+ \frac{2r_{n}+1}{2k^{'}c_{n}n}+ \dfrac{1}{2k^{'}n}\right\rbrace ,
$$
with $c_{n}:=\underset{ \left\lbrace\vert x\vert\leqslant L(9B+\frac{1}{n})\right\rbrace \cap\Theta }{\textrm{min}}\left\lbrace \frac{\psi^{''}(x)}{2} \right\rbrace.$
\end{theorem}

\begin{proof}
The proof of Theorem~\ref{th:LassoGLM} follows the same guidelines as the one for the Group Lasso. The main difference comes from concentration inequalities for the linear and log Laplace transform part of the loss function, leading to simpler bounds.
\end{proof}

This result extends the one of Bunea~\cite{BUN08} for logistic regression to generalized linear model and states convergence rates for the estimation and prediction error. The error bounds presented in Theorem \ref{th:LassoGLM} are of the same order as the ones stated by Bickel, Ritov and Tsybakov \cite{BIC09} in their analysis of the properties of the Lasso for standard linear models. We can also notice that in \cite{NEG09} Negahban and al. obtained bounds for the $l_{2}$-estimation error of the Lasso applied to GLMs of the same order as the one we get but under stronger conditions. In fact they assume that the distribution of the response $Y$ based on a predictor $X$ is given by 
$$P(Y \vert X; \beta^{*})=\exp(Y{\beta^{*}} ^{T}X-\psi({\beta^{*}} ^{T}X))$$ 
with $\vert X\vert\leq A$ and $\vert Y \vert\leq B$ and that $\psi^{''}$ is bounded from below on a suitable set. In our work we do not make these two last assumptions. The restricted eigenvalue property they use is also slightly stronger than $St(3,\varepsilon_{n},k)$. In addition we establish oracle inequalities for the prediction error.

The bounds in Theorem \ref{th:LassoGLM} are meaningful if $r_{n}$ is small (in particular if $n\gg \log(p)$) and $s^{*}$ is small. Indeed, the bound on the $l_{1}$-estimation error is of the order of $O\left(  s^{*}\sqrt{\frac{\log p }{n}}\right)  $. We can also notice that the minimum number of samples required to make the $l_{2}$-estimation and prediction error decrease to zero is of the order of $O\left( s^{*}\log p\right) $.

Then a relevant issue is the benefit of the Group Lasso over the Lasso. To understand better the power of a group structured estimator when the covariates have a group structure we refer to Huang and Zhang  \cite{HUANG10}. In this paper the authors investigate the benefit of sparsity with group structure compared to the usual Lasso. They develop a concept called strong group sparsity which means that the signal $\beta^{*}$ is efficiently covered by grouping. They showed that the Group Lasso is better than the usual Lasso for strongly group sparse data in the case of a standard linear model (see also \cite{LOU09}). Comparing Theorem \ref{theo-gp-glm} and Theorem \ref{th:LassoGLM} we see that an important improvement over the Lasso is given when the number of non-zero groups is much smaller than the total number of non-zero coefficients. Actually when the number of covariates increases this estimator removes completely the effect of the number of predictors. Indeed if we assume that the maximum size of the groups is finite then the   estimation error for the Group Lasso depends only on the total number of groups $G_{n}$ and on the number of significant groups $m^{*}$. Besides the prediction error for the Lasso is of the order of $O(s^{*}\frac{\log p}{n})$ whereas it is of the order of $O(m^{*}\frac{\log G_{n}}{n})$ for the Group Lasso. Therefore if $\beta^{*}$ has a group structure this is a meaningful improvement when $p$ is large and the number of groups $G_{n}$ is much more smaller. Notice this comparaison must be tempered by the fact that the two estimators require different conditions on the covariance matrix. In fact for a same data set there is not a condition which is weaker and implies the other. We also refer to the simulations in Section \ref{s:simu} that show some of the benefit of the Group Lasso over the Lasso when the covariates have a group structure in term of estimation and prediction error.

\section{\textbf{Applications and extensions}}
\label{s:extension}

\subsection{\textbf{Group Lasso for Poisson regression}}

Sparse logistic regression has been widely studied in the literature (see, for example, \cite{BUN08} and \cite{MEIR08}) but not the Poisson one for a sparse model. This last model is also very useful for many practical applications. For instance it is used to model count data and contingency tables. Poisson models are a special case of generalized linear models where the conditionnal law of $Y$ given  $X=x$ has a Poisson distribution with parameter $\lambda^{*}(x):=\exp({\beta^{*}}^{T}x)$. Therefore the conditional mean for Poisson regression is modeled by $\mathbb{E}(Y\vert X=x)=\exp({\beta^{*}}^{T}x)$. Thus the normalized function $\psi$ is the exponential function and is defined on $\mathbb{R}$. For this special link function we are going to specify the constants which appear in Theorem \ref{theo-gp-glm}. The log-likelihood based on the observations is given by
$$\mathcal{L}(\beta)=\sum_{i=1}^{n}\left[Y_{i}\beta^{T}X_{i}-\exp(\beta^{T}X_{i})-\log(Y_{i}!) \right] $$
and thus the Poisson loss function is defined by
$$l(\beta)=:l(\beta;x,y)=:-y\beta^{T}x+\exp(\beta^{T}x).$$
It is formula (\ref{loss-function}) with $\psi=\exp$. In the particular case of Poisson regression the conditional law is defined by $Y\vert X\sim \mathcal{P}\left(\lambda^{*}(X) \right) $ and the higher moments for a Poisson distribution are given by $$\mathbb{E}\left( Y^{k}\vert X\right)=\sum_{l=1}^{k}\left( \lambda^{*}(X)\right) ^{l}	S_{l:k}$$ 
where $S_{l:k}=\dfrac{1}{l!}\sum_{i=0}^{l}(-1)^{l-i}\binom{l}{i}i^{k}\geqslant 0$ is the number of partitions of a set with $l$ members into $k$ undersets. The number $\sum_{l=1}^{k}S_{l:k}:=B_{k}$ is called the $k^{th}$ Bell number and this number satisfies the relation $B_{n+1}=\sum_{k=0}^{n}\binom{k}{n}B_{k}$ (see for example \cite{SPI08}). So we can easily prove by induction that $B_{k}\leqslant k!$ for all $k\geqslant 1$.
Then we have on the event $\left\lbrace 0\leqslant \lambda^{*}(X)< 1\right\rbrace $
 $$\mathbb{E}\left(Y^{k}\right)=\mathbb{E}\left( \sum_{i=1}^{k}\left( \lambda^{*}(X)\right) ^{i}S_{i:k}\right)\leqslant  \sum_{i=1}^{k}S_{i:k}\leqslant k! $$
and on the event $\left\lbrace 1\leqslant\lambda^{*}(X)\right\rbrace  $ using (H.1) combined with (H.3) we find
$$\mathbb{E}\left(Y^{k}\right)= \sum_{i=1}^{k}S_{i:k}\mathbb{E}\left( \left( \lambda^{*}(X)\right) ^{i}\right) \leqslant k!(e^{LB})^{k} $$
Because $e^{LB}\geqslant 1$ we deduce that for all $k\geqslant 1$
$$\mathbb{E}\vert Y\vert ^{k}\leqslant k!(e^{LB})^{k}.$$
Thus for Poisson regression  we have $C_{L,B}=e^{LB}$.
Besides $\underset{\vert x\vert\leqslant L\kappa_{n}}{\textrm{max}}\left\lbrace \vert \psi^{'}(x)\vert \right\rbrace =e^{L\kappa_{n}}$ and $\underset{\vert x\vert\leqslant L(9B+\frac{1}{n})}{\textrm{min}}\left\lbrace \frac{\psi^{''}(x)}{2} \right\rbrace =\frac{1}{2}e^{-L(9B+\frac{1}{n})}$ where we recall that $\kappa_{n}:=17B+\frac{2}{n}$.
Therefore, in the case of Poisson regression, Theorem~\ref{theo-gp-glm} becomes 
\begin{corollary} \label{th:poissonGL}
Let $c_{n}:=\frac{1}{2}e^{-L(9B+\frac{1}{n})}$.
Assume that condition $GS(3,\frac{1}{2n},k)$ holds. If $$r_{n}\geqslant AKL e^{L(17B+\frac{2}{n})} \sqrt{\dfrac{\log(2G_{n})}{n}}$$ with $A>\sqrt{2}$ then, with probability at least $1-(C+2d_{\textrm{max}})(2G_{n})^{-A^{2}/2}$, we have

$$\Vert \hat{\beta}_n-{\beta^{*}} \Vert_{R}\leqslant \dfrac{4}{c_{n}k}r_{n}\gamma^{*}+ \left( 1+\frac{1}{r_{n}}\right) \frac{1}{2n},$$

$$\Vert \hat{\beta}_n-{\beta^{*}} \Vert_{2,1}\leqslant \dfrac{4}{c_{n}k\sqrt{d_{\textrm{min}}}}r_{n}\gamma^{*}+ \left( 1+\frac{1}{r_{n}}\right) \frac{1}{2n\sqrt{d_{\textrm{min}}}}$$
and
$$\mathbb{E}\left( \hat{\beta}_n^{T}X-{\beta^{*}}^{T}X\right) ^{2}\leqslant \dfrac{16}{c_{n}^{2}k}r_{n}^{2}\gamma^{*}+ \dfrac{2r_{n}+1}{2nc_{n}}. $$
Furthermore if condition $GS(2m^{*},3,\frac{1}{2n},k^{'})$ holds, then
$$\Vert \hat{\beta}_n-{\beta^{*}} \Vert_{2}^{2}\leqslant \dfrac{ d_{\textrm{max}}}{d_{\textrm{min}}}\left[ 160\dfrac{1}{{k^{'}}^{2}c_{n}^{2}}r_{n}^{2}\gamma^{*}+ 5\frac{2r_{n}+1}{k^{'}c_{n}n}+ 10\dfrac{1}{2k^{'}n}\right] .$$
\end{corollary}

\subsection{\textbf{Elastic net for generalized linear models}}

The most difficult part of the proof of Theorem \ref{theo-gp-glm} is to prove Proposition~\ref{prop:probaAetB}. Once this proposition has been proved it becomes easy to generalize the results presented above to any standard penalization using a modified version of the condition GS (depending on the norm we use). For example we can replace the $\ell_{1}$ norm by a combination of $\ell_{1}$ and $\ell_{2}$ norms. It is the so-called Elastic net introduced by Zou, Trevor and Hastie \cite{ZOU05} in the frame of linear regression. They showed that this estimator outperforms the Lasso in many situations for real world data and simulations. It is an alternative to the Group Lasso (the Elastic net has a behaviour similar to the one of the Group Lasso estimator). As the Lasso does, the Elastic net encourages sparsity and group selection but contrary to the Lasso when the sample size $n$ is smaller than $p$ the Elastic net can select more than $n$ significant variables. This estimator is the solution of a convex optimization problem. Zou, Trevor and Hastie in \cite{ZOU05} proposed an algorithm to solve this problem. The Elastic net estimator for the generalized linear model is defined by
\begin{equation}
\hat{\beta}_n=\underset{\beta\in \Lambda}{\textrm{argmin}}\left\lbrace  \mathbb{P}_{n}l(\beta)+2r_{n}\Vert \beta\Vert_{1} +t_{n}\Vert \beta\Vert_{2}^{2} \right\rbrace 
\label{eq elasticnet}
\end{equation}
where $r_{n}$ and $t_{n}$ are the penalty parameters. Theorem \ref{th:elasticnet} is an extension of the results first proved by Bunea \cite{BUN08} in the special case of logistic regression.
Let $2t_{n}B=r_{n}$. We have the following theorem
\begin{theorem}
\label{th:elasticnet}
Assume condition $St(4,\frac{1}{2n},k)$ holds. Let $$r_{n}\geqslant  AKL\left\lbrace  C_{L,B}\vee\underset{\left\lbrace\vert x\vert\leqslant L(17B+\frac{2}{n})\right\rbrace\cap\Theta}{\textrm{max}}\vert \psi^{'}(x)\vert\right\rbrace \sqrt{\dfrac{\log(2p)}{n}}$$ where $A>\sqrt{2}$ and $C_{L,B}$ depends only on $L$ and $B$. Then, 
with probability at least $1-C(2p)^{-A^{2}/2}$ (where $C$ is a universal constant), we have
$$\Vert \hat{\beta}_n-\beta^{*}\Vert_{1}\leqslant \dfrac{(2.5)^{2}}{t_{n}+c_{n}k}r_{n} s^{*}+\left( 1+\frac{1}{r_{n}}\right) \dfrac{1}{2n} $$ and
$$ \mathbb{E}\left( \hat{\beta}_n^{T}X-{\beta^{*}}^{T}X\right) ^{2}\leqslant  \dfrac{2(2.5)^{2}}{c_{n}k(t_{n}+c_{n}k)}r_{n}^{2}s^{*}+\dfrac{2r_{n}+3}{2c_{n}n}.
 $$
where $c_{n}:=\underset{\left\lbrace\vert x\vert\leqslant L(9B+\frac{1}{n})\right\rbrace \cap\Theta}{\textrm{min}}\left\lbrace  \frac{\psi^{''}(x)}{2}\right\rbrace $.
\label{theo-elastic}
\end{theorem}

We can notice that thanks to the $\ell_{2}$ penalty the bound for the $\ell_{1}$ and $\ell_{2}$ errors are less sensitive to small value of $k$ and small value of $c_{n}$ (which can appears when $L$ or  $B$ are large).

\section{\textbf{Simulations}}
\label{s:simu}

We are going to compare the performances of the Lasso and Group Lasso for Poisson Regression on simulated data sets. Computations have been performed using \texttt{R}. We use the package \texttt{grplasso} developped by Meir, van de Geer  and Buhlmann \cite{MEIR08} for the Group Lasso and the package \texttt{glmnet} developped by Friedman, Hastie and Tibshirani \cite{FRIED10} for the Lasso. The function \texttt{glmnet} fits the entire lasso regularization path for some generalized linear model via penalized maximum likelihood. We use this function in the particular case of Poisson regression.  The function \texttt{grplasso} fits the solution of a Group Lasso problem for a model of type \texttt{grpl.model} which generates models to be used for Group Lasso algorithm and identify the exponential family of the response and the link function which is used. Here we consider the function \texttt{PoissReg()} which generates a Poisson model. 

We simulate 100 data sets for each simulation  and we ran the Lasso and Group Lasso on these data sets. Each data set $X$ is cut into three separate subsets: a training data set, a validation data set and a test data set. We simulate responses via the model
$Y\sim \mathcal{P}(\exp(X\beta^{*}))$ where $\beta^{*}=({\beta^{*}}^{1},...,{\beta^{*}}^{G})$ with $g$ groups with non zero coefficients among the $G$ groups. The training data set ($X_{\textrm{train}}$ of size $n_{\textrm{train}}$) is used to fit the model (we estimate the target $\beta^{*}$ for a sequence of the tuning parameter $\lambda$ and denote by $\beta_{\lambda}$ the estimate of $\beta^{*}$ obtained for such a parameter for the Lasso and the Group Lasso estimator). Then, we use the validation data ($X_{\textrm{valid}}$ of size $n_{\textrm{valid}}$) to evaluate the performance of the fitted model according to a specific loss function. We define the optimal tuning parameter as the one for which the deviation from the fitted mean to the response is minimal i.e. $\lambda_{\textrm{opt}}\in \underset{\lambda}{\textrm{argmin}}\left\lbrace \dfrac{1}{n_{\textrm{valid}}}\Vert Y-\exp(X_{\textrm{valid}}\beta_{\lambda})\Vert_{2}^{2} \right\rbrace $.  From that we determine the model with the parameter vector $\beta_{\lambda_{\textrm{opt}}}$. Then we compute the hits i.e the number of correctly identified relevant variables, the false positives i.e the number of non significant variables choosen as relevant and the degreee of freedom i.e the total number of variables selected in the model. Finally, we estimate the performance of the selected model by computing the coefficients estimation error $\Vert \beta^{*}-\beta_{\lambda_{\textrm{opt}}}\Vert_{1}$ and the prediction error $ \Vert X_{\textrm{test}}\beta^{*}-X_{\textrm{test}}\beta_{\lambda_{\textrm{opt}}}\Vert_{2}$ on the test data ($X_{\textrm{test}}$ of size $n_{\textrm{test}}$). We ran the Lasso and Group Lasso on these data sets.

Eight models are considered in the simulations. For each simulation $n_{train}=50$, $n_{valid}=50$, $n_{test}=100$. To compare the Lasso and Group Lasso we use a random design matrix where the predictors are simulated as followed according to a uniform distribution to have bounded predictors.
\begin{enumerate}
\item 
\begin{itemize}
\item $X_{i}=U_{1}+\varepsilon_{i}$ for $1\leqslant i\leqslant 10$ with $U_{1}\sim U(\left[0,1 \right] )$
\item $X_{i}=U_{2}+\varepsilon_{i}$ for $11\leqslant i \leqslant 20$ with $U_{2}\sim U(\left[0,1 \right] )$
\item $X_{i}=U_{i}$ for the last $100$ variables $U_{i}\sim U(\left[-0.1,0.1 \right] )$
\end{itemize}
with $\varepsilon_{i}$ i.i.d $\sim U(\left[0,0.01 \right] )$.\\
The covariates within the first two blocks are highly correlated ($\sim 0.8$) and there are small correlations between the blocks.  The target is 
$$\beta^{*}=(\underbrace{0.3,...,0.3}_{10},\underbrace{0.2,...,0.2}_{10},\underbrace{\underbrace{0,...,0}_{10},...,\underbrace{0,...,0}_{10}}_{10}).$$
\item
The simulation is the same as the first one except that $\varepsilon_{i}$ i.i.d $\sim U(\left[0,1 \right] )$. Thus there are small correlations within and between groups ($\sim 0.5$). 

\item 
The simulation is the same as the first one except that $\varepsilon_{i}$ i.i.d $\sim U(\left[0,1.2 \right] )$. Thus there are very small correlations within and between groups ($\sim 0.2$). \\\\
For all the following simulations the non zero groups are generated in the same way as in the second example. For  $j=1,...,m^{*}$ and $i=1,...,d_{j}$, $X_{i}=U_{j}+\varepsilon_{i}$ where $U_{j}\sim U(\left[0,1 \right] )$  and $\varepsilon_{i}$ i.i.d $\sim U(\left[0,1 \right] )$. The non influencial groups are simulated according to $X_{i}=U_{i}$ with $U_{i}\sim U(\left[  -0.1,0.1 \right] )$.
In the two following simulations we increase the size of the non zero groups
\item 
 $$\beta^{*}=(\underbrace{0.3,...,0.3}_{20},\underbrace{0.2,...,0.2}_{20},\underbrace{\underbrace{0,...,0}_{20},...,\underbrace{0,...,0}_{20}}_{10}).$$
\item 
$$\beta^{*}=(\underbrace{0.3,...,0.3}_{5},\underbrace{0.2,...,0.2}_{5},\underbrace{\underbrace{0,...,0}_{5},...,\underbrace{0,...,0}_{5}}_{10}).$$
In the last three simulations it is the number of the non zero groups which is increased.
\item
$$\beta^{*}=(\underbrace{0.2,...,0.2}_{10},\underbrace{0.2,...,0.2}_{10},\underbrace{\underbrace{0,...,0}_{10},...,\underbrace{0,...,0}_{10}}_{10}).$$
\item 
$$\beta^{*}=(\underbrace{\underbrace{0.2,...,0.2}_{10},...,\underbrace{0.2,...,0.2}_{10}}_{4},\underbrace{\underbrace{0,...,0}_{10},...,\underbrace{0,...,0}_{10}}_{10}).$$
\item
$$\beta^{*}=(\underbrace{\underbrace{0.2,...,0.2}_{10},...,\underbrace{0.2,...,0.2}_{10}}_{6},\underbrace{\underbrace{0,...,0}_{10},...,\underbrace{0,...,0}_{10}}_{10}).$$
\end{enumerate}

The results of the eight simulations are reported in Table I hereafter where $p$ is the number of covariates, $s^{*}$ the number of significant covariates, $G$ the number of groups, $m^{*}$ the number of non zero groups and $v$ is the size of the non zero groups. The Group Lasso seems to perform better than the Lasso to include the relevant predictors into the model particularly when there are high within group correlations. The Lasso tends to select fewer variables among the influencial ones (the Lasso selects only some variables from the groups of highly correlated predictors) than the Group Lasso does in the case of highly correlated covariates and when the size or the number of the nonzero groups is large. The Group Lasso succeeds in including the true significant groups in most of the cases. On the contrary the Group Lasso estimator tends to add more irrelevant covariates than the Lasso does in particular when the number or the size of the non influencial groups is large. We can expect such a result because the Group Lasso estimator includes not single variable one after another but groups of variables (when one of the covariates is included in the model all the others which belongs to the same group are also included in the model). Thus the Group Lasso selects models that are larger than the true model. However when the Group Lasso selects a number of  covariates which is the same than the number of significants covariates it is, with high probability, the correct groups which are included. We have also measured the performances of the Lasso and Group Lasso in terms of estimation error and prediction error. The Group Lasso seems to perform better than the Lasso in term of estimation error and of prediction error in most of the cases and the improvement is particularly meaningful for prediction error. We can also notice that the performances of the two estimators decreases with the increase of $G$ and $m^{*}$. To conclude, the Group Lasso estimator seems to perform better than the Lasso to include the hits in the model and in terms of prediction and estimation error when the covariates are structured into groups and in particular in the case of high correlations within groups.

\begin{table*}[h]
\begin{tabular}{|c|c|c|c|c|c|c|c|c|}
\hline
Simulation &1 &2&3&4&5&6&7&8\\
\hline 
$s^{*}/p$  & $20/120$ & $20/120$  & $20/120$ & $40/240$& $10/60$& $20/120$ & $40/140$ &$60/160$ \\ 

$G$ & $12$ & $12$  & $12$ & $12$ & $12$& $12$ & $14$& $16$\\ 
 
$m^{*}$ & $2$ & $2$  & $2$ & $2$  & $2$& $2$ & $4$ & $6$\\ 
 
$v$ & $10$ & $10$  & $10$ & $20$ & $5$& $10$ & $10$ & $10$ \\ 
\hline 
$\textrm{Mean Hit lasso (\%) }$ & $86,9$ & $99,4$  & $99,95$ & $66,27$ &$87,6$ & $95,9$ &$78,79$&$54,16$\\ 
 
$\textrm{Mean Hit group lasso (\%) }$ & $100$ & $100$  & $100$ & $98,5$ & $100$& $100$ & $100$& $98,8$\\ 
\hline 
$\textrm{Mean False positive lasso (\%)}$ & $15,02$ & $10,61$ & $9,25$ & $2,39$ & $47,38$  & $18,96$ & $9,33$ & $8,87$\\ 

$\textrm{Mean False positive  group lasso (\%) }$ & $28,7$ & $36,1$  & $33,9$ & $74$ &$34,2$& $29,6$ & $61,9$ & $37,9$ \\ 
\hline 
$\textrm{Mean Nonzero lasso}$ & $32,4$ & $30,49$  & $29,24$ & $31,29$& $32,45$ & $38,14$ &$40,85$ &$41,37$\\ 
 
$\textrm{Mean Nonzero gp lasso}$ & $48,7$& $56,1$  & $53,9$ & $187,4$& $27,1$ & $49,6$ & $101,9$ & $97,2$\\ 

\hline 
$\textrm{Mean Prediction error lasso}$  & $0,19$ & $0,15$  & $0,18$ & $6,73$ & $0,15$& $0,17$ & $3,22$& $17,97$\\ 
 
$\textrm{Mean Prediction error gp lasso}$ & $0,021$ & $0,007$  & $0,004$ & $1,98$ & $0,03$& $0,016$ & $0,012$& $1,32$\\ 

\hline 
$\textrm{Mean Estimation error lasso}$  & $6,36$ & $2,84$ & $2,08$ & $9,79$ & $15,31$& $6,37$ & $8,45$ & $16,28$  \\ 
 
$\textrm{Mean Estimation error gp lasso}$ & $5,54$ & $2,66$ & $1,64$ & $18,65$ & $4,22$& $3,40$ & $2,77$ & $12,76$ \\ 
\hline 
\end{tabular}
\label{tab:result-simu}
\caption{Results of the simulations for the eight models}
\end{table*}

To better understand the difference of behaviour between the Lasso and the Group Lasso in terms of variables selection we also provide some kind of ROC curves. These curves are created by plotting the fraction of true positives (of the estimated parameter) out of the positive coefficients of $\beta^{*}$ for discretized values of the penalty parameter going from zero (completely dense solution) up to the value where the first group of covariates enters the model (completely sparse solution). To ensure that we well detect each new inclusion of covariates, the discretization involves $10000$ values of the penalty parameter. We plot the curves for each model considered. Three models are described below and the respective ROC curves are shown in Figures \ref{fig:1}, \ref{fig:2} and \ref{fig:3}.

For each model we have
\begin{itemize}
\item for  $j=1,...,m^{*}$ and $i=1,...,d_{j}$, $X_{i}=U_{j}+\varepsilon_{i}$ where $\varepsilon_{i}$ i.i.d $\sim U(\left[-1,1 \right] )$ and $U_{j}\sim U(\left[0,1 \right] )$ for simulation 1 and $U_{j}\sim U(\left[-1,1 \right] )$ for simulation 2 and 3 .
\item for the non significant groups $X_{i}=U_{i}$ with $U_{i}\sim U(\left[  -0.1,0.1 \right] )$ for simulation 1 and 2 and $U_{i}\sim U(\left[  -0.01,0.01 \right] )$ for simulation 3.
\end{itemize}
  
\begin{enumerate}
\item For the first simulation the target is \\
$\beta^{*}=(\underbrace{\underbrace{0.2,...,0.2}_{5},...\underbrace{0.2,...,0.2}_{5}}_{10},\underbrace{\underbrace{0,...,0}_{5},...,\underbrace{0,...,0}_{5}}_{50}).$
\begin{figure}[H]
\begin{center}
\includegraphics[scale=0.4]{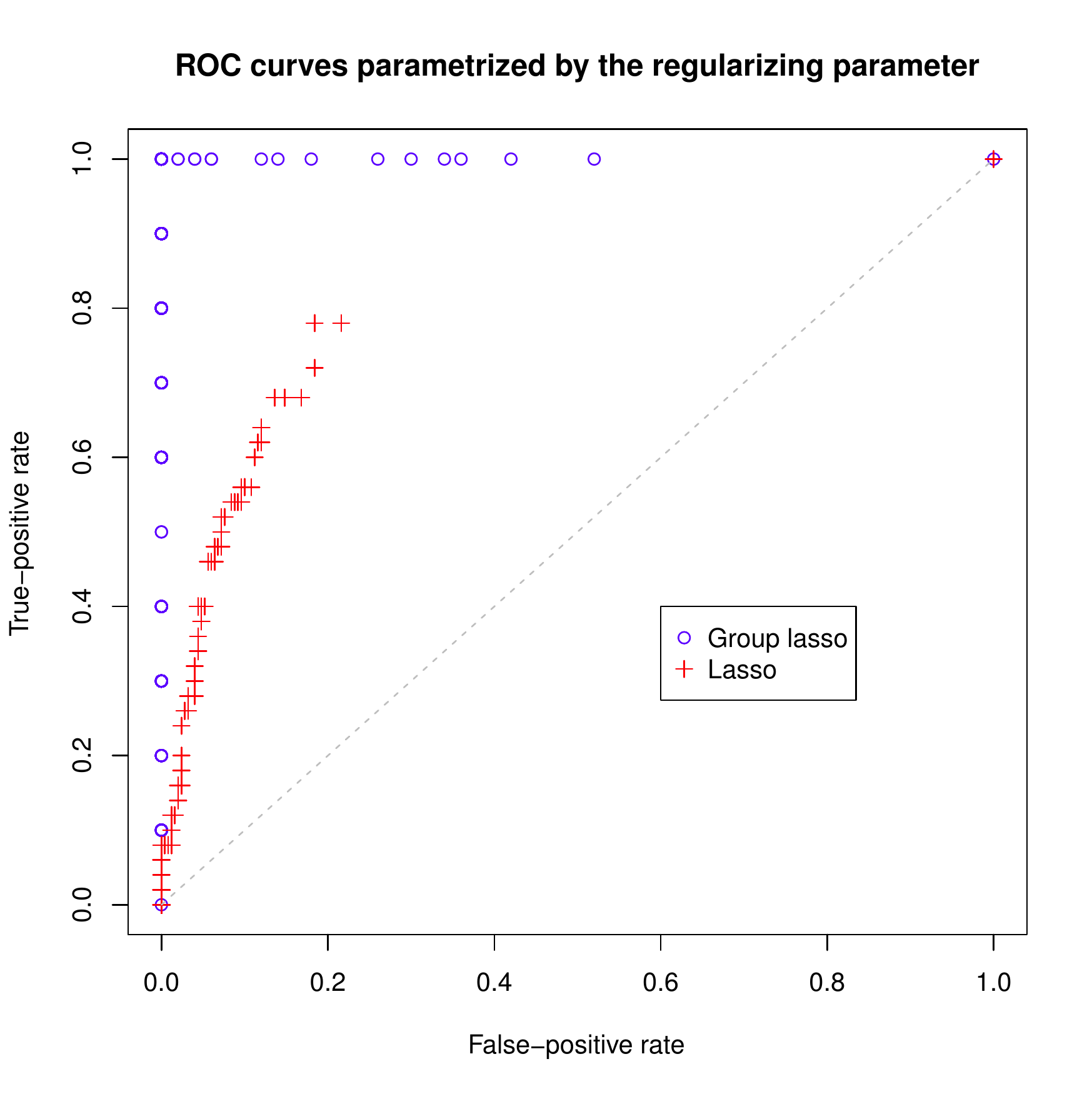}
\caption{Simulation 1}
\label{fig:1}
\end{center}
\end{figure}

\item For the second simulation the target is \\
$\beta^{*}=(\underbrace{\underbrace{0.2,...,0.2}_{5},...\underbrace{0.2,...,0.2}_{5}}_{10},\underbrace{\underbrace{0,...,0}_{5},...,\underbrace{0,...,0}_{5}}_{100}).$
\begin{figure}[H]
\begin{center}
\includegraphics[scale=0.4]{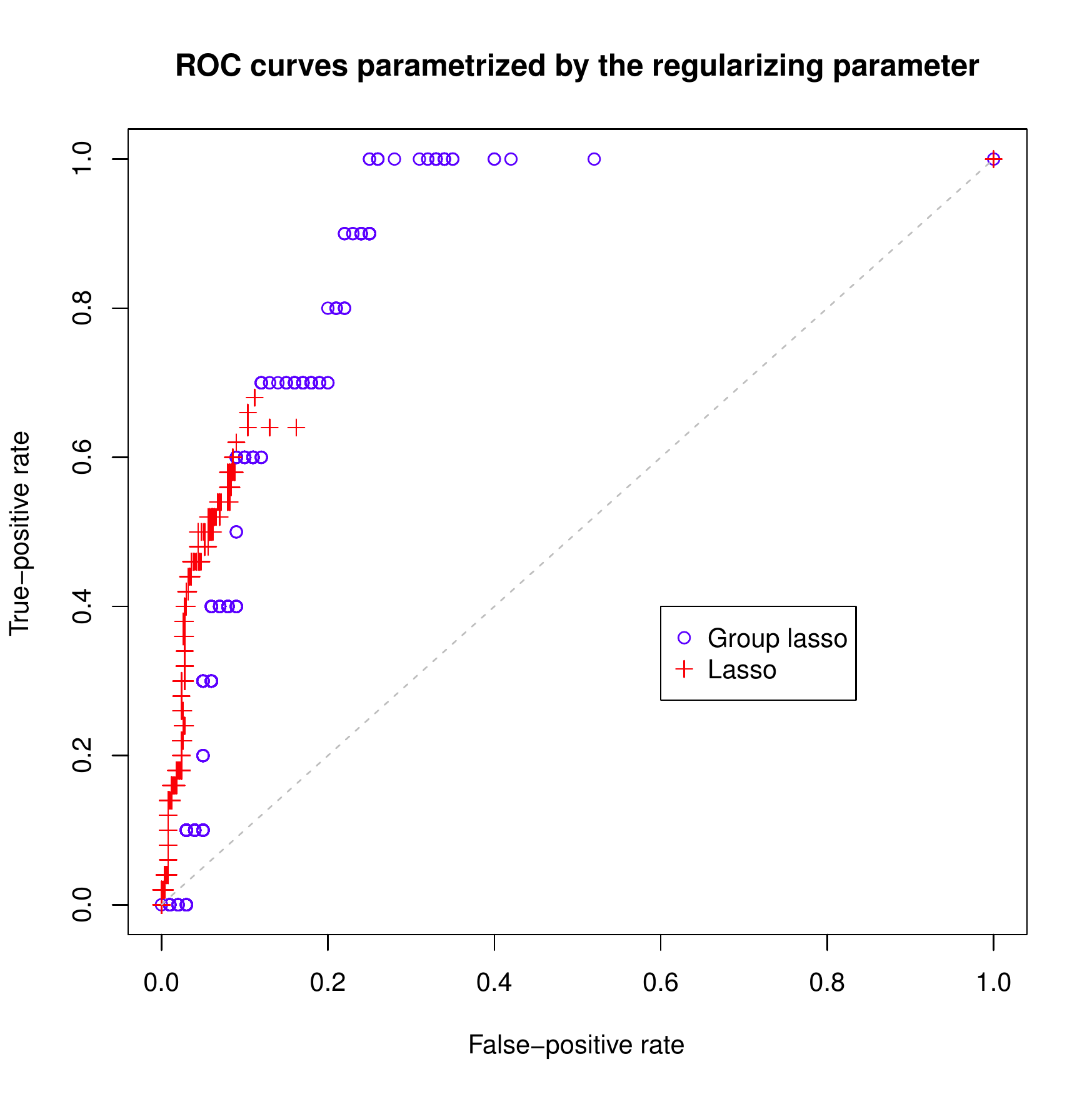}
\caption{Simulation 2}
\label{fig:2}
\end{center}
\end{figure}

\item For the third simulation the target is \\
$\beta^{*}=(\underbrace{\underbrace{0.2,...,0.2}_{5},...\underbrace{0.2,...,0.2}_{5}}_{10},\underbrace{\underbrace{0,...,0}_{5},...,\underbrace{0,...,0}_{5}}_{100}).$
\begin{figure}[H]
\begin{center}
\includegraphics[scale=0.4]{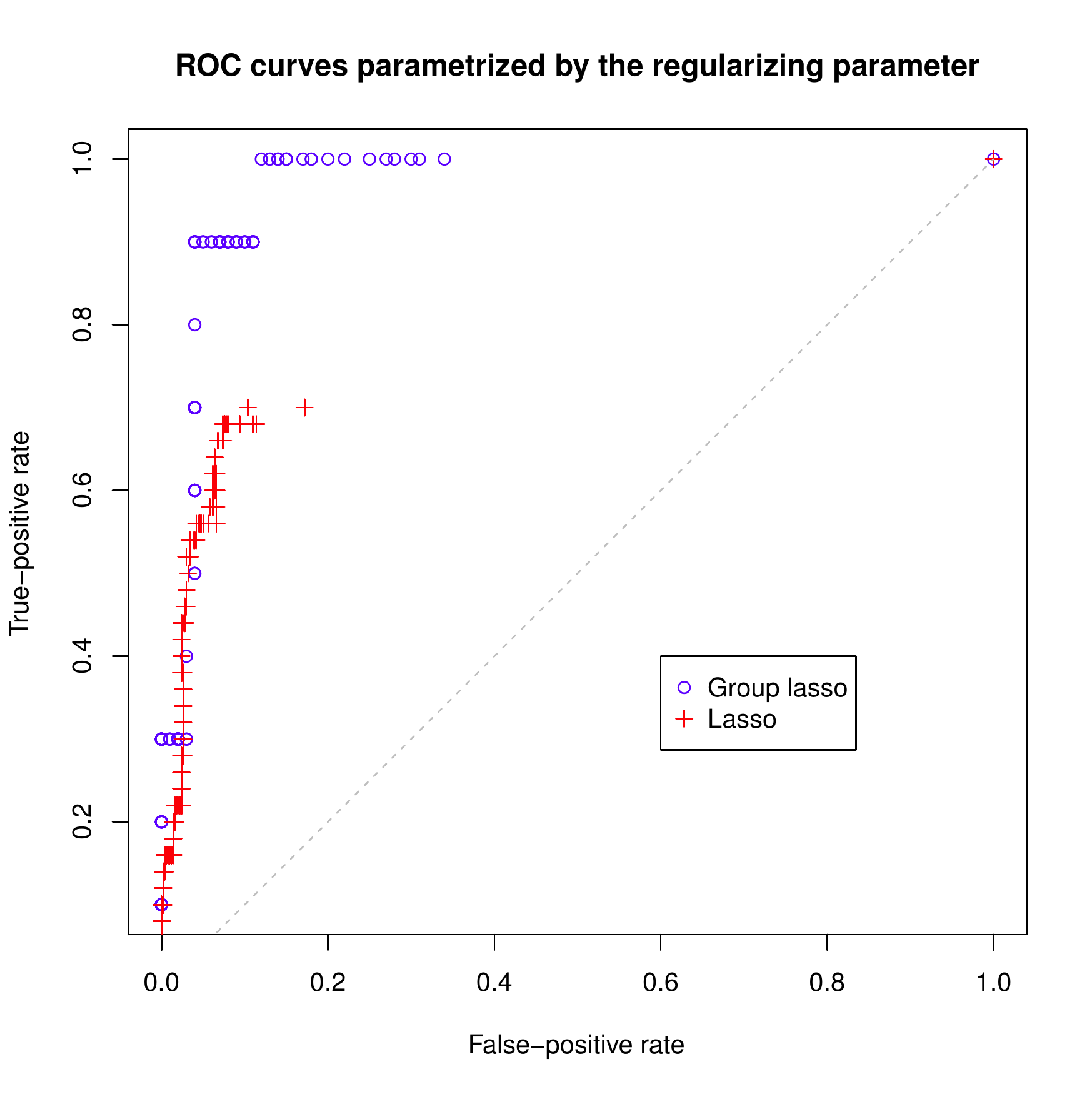}
\caption{Simulation 3}
\label{fig:3}
\end{center}
\end{figure}
\end{enumerate}

The simulations clearly illustrate that contrary to the Lasso the Group Lasso includes not single covariates one after another but all the covariates which belong to the same group at the same time. The Group Lasso tends to include in priority the significant groups and then the non significant groups in such a way that for some values of the penalty parameter the rate of perfect selection can be very close to one and even equal to one in some particular cases (see for instance Figure \ref{fig:1}). We can notice that the ROC curves for the Lasso follow at the beginning those of the Group Lasso and then fall below. The break in the curves of the Lasso illustrate a special feature of this estimator. Actually  the number of covariates included in the model for the Lasso is restricted by the sample size whereas for the Group Lasso the restriction is given by the number of groups. That is why for the Lasso the number of true positives but also the number of false positives are smaller than those of the Group Lasso (once approximately $100$ covariates are added in the model the Lasso stop including other ones). We can also notice that the Group Lasso is less stable than the Lasso in terms of variables selection. In fact a little change in the penalty parameter can substantially increase the number of false positives. This particular feature of the Group Lasso due to a group structure penalty is reflected in Table I with a rate of false positives larger for the Group Lasso in most of the cases. Thus Figures \ref{fig:1}, \ref{fig:2} and \ref{fig:3} confirm what has been previously deduced from the results presented in Table I. To conclude the Group Lasso seems to be more reliable than the Lasso to include the variables of interest but in return it tends to incorporate more false positives.

\section{\textbf{Conclusion}}

We consider the generalized linear model in high dimensional settings and use the Group Lasso estimator to estimate the regression parameter $\beta^{*}$ when the covariates are naturally structured into groups of variables and the true parameter is group sparse (just a few variables are relevant to explain the response). Under some assumptions on the sparsity of $\beta^{*}$, on the correlations between the groups of covariates and on the tuning parameter of the estimator we state general oracle inequalities for estimation and error prediction for the Group Lasso estimator applied to generalized linear models. In the particular case of groups of size one, we provide original inequalities for the Lasso in the case of generalized linear models extending the results of Bunea~\cite{BUN08} for logistic regression. Furthermore, we extend these results to other penalties such as the Elastic net. Then we compare the performances of the Lasso to the ones of the Group Lasso on simulated data. We show the improvement in terms of variables selection and prediction errror of the Group Lasso compared to the Lasso when the covariates are structured into groups. Moreover we illustrate on these simulated data the impact of the total number of groups, of the number of non zero groups and of the size of the groups on the performances of the Group Lasso. The conclusion is that the Group Lasso estimator behave well in a high dimensional setting under sparsity and group correlations assumptions. However the main
drawback of the Group Lasso is that we need an a priori knowledge on the groups and it is not always possible.


%


\appendices
\section{\textbf{Proof of Theorem \ref{theo-gp-glm}}}

\subsection{\textbf{The main steps of the proof}}

The proof follows the guidelines  in \cite{BUH11} or \cite{LOU02}. Using the mere definition of $\hat{\beta}_n$, we have
\begin{equation}
\mathbb{P}_{n}l(\hat{\beta}_n)+2r_{n}\sum_{g=1}^{G_{n}}\sqrt{d_{g}}\Vert {\hat{\beta}_n}^{g}\Vert_{2} \leqslant  \mathbb{P}_{n}l(\beta^{*}) +2r_{n}\sum_{g=1}^{G_{n}}\sqrt{d_{g}}\Vert {\beta^{*}}^{g}\Vert_{2} .
\label{eq debut-gpl}
\end{equation}
Hence we get
$$\mathbb{P}\left( l(\hat{\beta}_n)-l(\beta^{*})\right)+2r_{n}\sum_{g=1}^{G_{n}}\sqrt{d_{g}}\Vert \hat{\beta}_n^{g}\Vert_{2}$$
\begin{equation}
\leqslant \left( \mathbb{P}_{n}-\mathbb{P}\right) \left( l(\beta^{*})-l(\hat{\beta}_n)\right) +2r_{n}\sum_{g=1}^{G_{n}}\sqrt{d_{g}}\Vert {\beta^{*}}^g\Vert_{2}.
\label{eq-principale-gpl}
\end{equation}
We decompose the empirical process into a linear part and a part which depends on the normalized parameter $\psi$.$$\left( \mathbb{P}_{n}-\mathbb{P}\right) \left( l(\beta^{*})-l(\hat{\beta}_n)\right)$$ $$=\left( \mathbb{P}_{n}-\mathbb{P}\right) \left( l_{l}(\beta^{*})-l_{l}(\hat{\beta}_n)\right)+\left( \mathbb{P}_{n}-\mathbb{P}\right) \left( l_{\psi}(\beta^{*})-l_{\psi}(\hat{\beta}_n)\right)$$ where $$l_{l}(\beta):=l_{l}(\beta,x,y)=-y\beta^{T}x$$ and $$l_{\psi}(\beta):=l_{\psi}(\beta,x)=\psi(\beta^{T}x).$$
From Proposition~\ref{prop:upbound1} and Proposition~\ref{prop:upbound2} and by adding $r_{n}\sum_{g=1}^{G_{n}}\sqrt{d_{g}}\Vert \hat{\beta}_n^{g}-{\beta^{*}}^g\Vert_{2}$ to both sides of the inequality (\ref{eq-principale-gpl}) we find, on $\mathcal{A}\cap \mathcal{B}$, that  
$$r_{n}\sum_{g=1}^{G_{n}}\sqrt{d_{g}}\Vert \hat{\beta}_n^{g}-{\beta^{*}}^g\Vert_{2}+\mathbb{P}\left( l(\hat{\beta}_n)-l(\beta^{*})\right)$$
$$\leqslant  2r_{n}\sum_{g=1}^{G_{n}}\sqrt{d_{g}}\left( \Vert \hat{\beta}_n^{g}-{\beta^{*}}^g\Vert_{2}+\Vert {\beta^{*}}^g\Vert_{2}-\Vert \hat{\beta}_n^{g}\Vert_{2}\right) +\frac{r_{n}}{2}\varepsilon_{n}.$$
If $g\notin H^{*}$ then $\Vert \hat{\beta}_n^{g}- {\beta^{*}}^{g}\Vert_{2} +\Vert {\beta^{*}}^{g}\Vert_{2}  - \Vert \hat{\beta}_n^{g}\Vert=0$ and otherwise $ \Vert {\beta^{*}}^{g}\Vert_{2}  - \Vert \hat{\beta}_n^{g}\Vert_{2}\leqslant \Vert \hat{\beta}_n^{g}- {\beta^{*}}^{g}\Vert_{2}$. So the last inequality can be bounded by
\begin{equation}
 4r_{n}\sum_{g\in H^{*}}\sqrt{d_{g}}\Vert \hat{\beta}_n^{g}-{\beta^{*}}^g\Vert_{2}+\frac{r_{n}}{2}\varepsilon_{n}.
\label{eq-principale2-gpl}
\end{equation}
By the definition of $\beta^{*}$ we have $\mathbb{P}\left( l(\hat{\beta}_n)-l(\beta^{*})\right)>0$ and therefore
$$ \sum_{g\notin H^{*}}\sqrt{d_{g}}\Vert \hat{\beta}_n^{g}-{\beta^{*}}^g\Vert_{2}\leqslant3\sum_{g\in H^{*}}\sqrt{d_{g}}\Vert \hat{\beta}_n^{g}-{\beta^{*}}^g\Vert_{2} +\frac{\varepsilon_{n}}{2}$$
i.e $\hat{\beta}_n-\beta^{*} \in S(3,\dfrac{\varepsilon_{n}}{2})$.
The next proposition provides a lower bound for $\mathbb{P}\left( l(\hat{\beta}_n)-l(\beta^{*})\right)$.
\begin{prop}\label{prop:lowbound}
On the event $\mathcal{A}\cap \mathcal{B}$
we have $$\mathbb{P}\left( l(\hat{\beta}_n)-l(\beta^{*})\right)
\geqslant c_{n}\mathbb{E}\left[ \left( f_{\hat{\beta}_n}(X)-f_{\beta^{*}}(X)\right) ^{2}\right]$$
with $c_{n}:=\underset{\vert x\vert\leqslant L(9B+\frac{1}{n})}{\textrm{min}}\left\lbrace  \frac{\psi^{''}(x)}{2}\right\rbrace $.
\end{prop}
\begin{proof}
$$\mathbb{P}\left( l(\hat{\beta}_n)-l(\beta^{*})\right)=-\mathbb{E}\left[ \mathbb{E}(Y\vert X)\left( f_{\hat{\beta}_n}(X)-f_{\beta^{*}}(X)\right)\right]$$
$$+\mathbb{E}\left[ \psi^{'}(f_{\beta^{*}}(X))\left( f_{\hat{\beta}_n}(X)-f_{\beta^{*}}(X)\right) \right] $$
$$+\mathbb{E}\left[ \frac{\psi^{''}(f_{\tilde{\beta}}(X))}{2}\left( f_{\hat{\beta}_n}(X)-f_{\beta^{*}}(X)\right)^{2} \right] $$
where $\tilde{\beta}^{T}X$ is an intermediate point between $\hat{\beta}_n^{T}X$ and ${\tilde{\beta}}^{T}X$ given by a second order Taylor expansion of $\psi$. Since $\psi^{'}(f_{\beta^{*}}(X))=\mathbb{E}(Y\vert X)$ we find
$$\mathbb{P}\left( l(\hat{\beta}_n)-l(\beta^{*})\right)=\mathbb{E}\left[ \frac{\psi^{''}(f_{\tilde{\beta}}(X))}{2}\left( f_{\hat{\beta}_n}(X)-f_{\beta^{*}}(X)\right) ^{2}\right]. $$
Besides we have
$$\vert {\tilde{\beta}}^{T}X\vert \leqslant \vert \tilde{\beta}^{T}X-{\beta^{*}}^{T}X\vert+\vert {\beta^{*}}^{T}X\vert $$
$$\leqslant \sum_{g=1}^{G_{n}}\Vert {\beta^{*}}^{g}-\beta^{g} \Vert_{2}\Vert X^{g}\Vert_{2}+\sum_{g=1}^{G_{n}}\Vert {\beta^{*}}^{g} \Vert_{2}\Vert X^{g}\Vert_{2}.$$
Applying (H.1), we find $$\Vert X^{g}\Vert_{2}\leqslant L\sqrt{d_{g}}.$$ 
Then using Lemma~\ref{local} and (H.3) we find
\begin{equation}
\vert {\tilde{\beta}}^{T}X\vert \leqslant  LM+LB\quad \rm{a.s.}
\label{eq:majo-taylor}
\end{equation}
Furthermore, $\beta^{*}$ and $\hat{\beta}_{n}$ belongs to $\Lambda$ which is a convex set. Therefore $\tilde{\beta} \in \Lambda$ and ${\tilde{\beta}}^{T}X \in \Theta$ a.s.
Thus we conclude $$\mathbb{P}\left( l(\hat{\beta}_n)-l(\beta^{*})\right)
\geqslant c_{n}\mathbb{E}\left[ \left( f_{\hat{\beta}_n}(X)-f_{\beta^{*}}(X)\right) ^{2}\right]$$
where $c_{n}:=\underset{\left\lbrace \vert x\vert\leqslant L(M+B)\right\rbrace  \cap\Theta}{\textrm{min}}\left\lbrace  \frac{\psi^{''}(x)}{2}\right\rbrace  $.
\end{proof}

From Propositon~\ref{prop:lowbound} and (\ref{eq-principale2-gpl}) we deduce that
$$r_{n}\sum_{g=1}^{G_{n}}\sqrt{d_{g}}\Vert \hat{\beta}_n^{g}-{\beta^{*}}^g\Vert_{2}+c_{n}\mathbb{E}\left( \hat{\beta}_n^{T}X-{\beta^{*}}^{T}X\right) ^{2}$$
\begin{equation}
\leqslant 4r_{n}\sum_{g\in H^{*}}\sqrt{d_{g}}\Vert \hat{\beta}_n^{g}-{\beta^{*}}^g\Vert_{2}+\frac{r_{n}}{2}\varepsilon_{n}.
\label{eq-principale2bis-gpl}
\end{equation}

Then the end of the proof is similar to the end of the proof of Theorem 2.4 in \cite{BUN08}. Let $\Sigma$ be the $p \times p$ covariance matrix whose entries are $\mathbb{E}(X_{k}X_{j})$. We have $$\mathbb{E}\left( \hat{\beta}_n^{T}X-{\beta^{*}}^{T}X\right) ^{2}=(\hat{\beta}_n-{\beta^{*}})^{T}\Sigma(\hat{\beta}_n-{\beta^{*}}).$$  Because condition $GS(3,\frac{\varepsilon_{n}}{2},k)$ is satisfied we have $$c_{n}(\hat{\beta}_n-{\beta^{*}})^{T}\Sigma(\hat{\beta}_n-{\beta^{*}})\geqslant c_{n}k\sum_{g \in H^{*}}\Vert \hat{\beta}_n^{g}-{\beta^{*}}^g\Vert_{2}^{2}-\dfrac{\varepsilon_{n}}{2}.$$
Then by using Cauchy-Schwarz inequality in (\ref{eq-principale2bis-gpl}) we find
$$r_{n}\sum_{g=1}^{G_{n}}\sqrt{d_{g}}\Vert \hat{\beta}_n^{g}-{\beta^{*}}^g\Vert_{2}+c_{n}k\sum_{g \in H^{*}}\Vert \hat{\beta}_n^{g}-{\beta^{*}}^g\Vert_{2}^{2}$$
$$\leqslant 4r_{n}\sqrt{\sum_{g \in H^{*}}d_{g}}\sqrt{\sum_{g \in H^{*}}\Vert \hat{\beta}_n^{g}-{\beta^{*}}^{g}\Vert_{2}^{2}}+(r_{n}+1)\dfrac{\varepsilon_{n}}{2}.$$
Now the fact that $2xy\leqslant tx^{2}+y^{2}/t$ for all $t>0$ leads to the following inequality
$$r_{n}\sum_{g=1}^{G_{n}}\sqrt{d_{g}}\Vert \hat{\beta}_n^{g}-{\beta^{*}}^g\Vert_{2}+c_{n}k\sum_{g \in H^{*}}\Vert \hat{\beta}_n^{g}-{\beta^{*}}^g\Vert_{2}^{2}$$
\begin{equation}
\leqslant 4tr_{n}^{2}\gamma^{*}+\frac{1}{t}\sum_{g \in H^{*}}\Vert \hat{\beta}_n^{g}-{\beta^{*}}^g\Vert_{2}^{2}+(r_{n}+1)\dfrac{\varepsilon_{n}}{2}.
\label{equation intermed}
\end{equation}
Replacing $t$ by $\dfrac{1}{c_{n}k}$ in (\ref{equation intermed}) we obtain
$$\sum_{g=1}^{G_{n}}\sqrt{d_{g}}\Vert \hat{\beta}_n^{g}-{\beta^{*}}^g \Vert_{2}\leqslant \dfrac{4}{c_{n}k}r_{n}\gamma^{*}+ (1+\frac{1}{r_{n}})\dfrac{\varepsilon_{n}}{2}.$$
i.e $$\Vert \hat{\beta}_n-{\beta^{*}} \Vert_{R} \leqslant \dfrac{4}{c_{n}k}r_{n}\gamma^{*}+ (1+\frac{1}{r_{n}})\dfrac{\varepsilon_{n}}{2}.$$
What is more $$\sqrt{d_{\textrm{min}}}\Vert \hat{\beta}_n-{\beta^{*}} \Vert_{2,1}\leqslant \sum_{g=1}^{G_{n}}\sqrt{d_{g}}\Vert \hat{\beta}_n^{g}-{\beta^{*}}^g \Vert_{2}.$$
This yields 
\begin{equation}
\label{eq:esti-error}
\Vert \hat{\beta}_n-{\beta^{*}} \Vert_{2,1}\leqslant \dfrac{4}{c_{n}k\sqrt{d_{\textrm{min}}}}r_{n}\gamma^{*}+ \left( 1+\frac{1}{r_{n}}\right) \frac{1}{2n\sqrt{d_{\textrm{min}}}}.
\end{equation}

A similar argument could be made to prove
\begin{equation}
\label{eq:predict-error}
\mathbb{E}\left( \hat{\beta}_n^{T}X-{\beta^{*}}^{T}X\right) ^{2}\leqslant \dfrac{16}{c_{n}^{2}k}r_{n}^{2}\gamma^{*}+ \dfrac{2r_{n}+1}{c_{n}}\dfrac{\varepsilon_{n}}{2}. 
\end{equation}
Finally we conclude the proof using Proposition~\ref{prop:probaAetB}.

\subsection{\textbf{Proof of Proposition~\ref{prop:probaAetB}}}

\begin{proof}
Let $A>\sqrt{2}$. We recall that we have made the assumption $\frac{\log(2G_{n})}{n}\leqslant 1$. We deduce Proposition~\ref{prop:probaAetB} from the two following lemmas.
\begin{lem}\label{lem:propA}
Let $$r_{n}\geqslant \left( 8\sqrt{2}ALC_{L,B}\sqrt{\dfrac{\log(2G_{n})}{n}}\right) \vee \left( 16A^{2}LC_{L,B}\dfrac{\log(2G_{n})}{n}\right)$$
with $A>1$.
Then $$\mathbb{P}\left\lbrace \mathcal{A}\right\rbrace \geqslant 1-2d_{\textrm{max}}(2G_{n})^{1-A^{2}}.$$
\end{lem}

\begin{lem}
\label{lem:propB}
Let $$r_{n}\geqslant 20AL\left( \underset{\left\lbrace \vert x\vert\leqslant L\kappa_{n}\right\rbrace \cap\Theta}{\textrm{max}}\vert \psi^{'}(x)\vert\right)\sqrt{\dfrac{2\log 2G_{n}}{n}}$$ where $A\geqslant 1$. Then  $$\mathbb{P}(\mathcal{B})\geqslant 1-C(2G_{n})^{-A^{2}/2}$$
where we recall $\kappa_{n}:=17B+\dfrac{2}{n}$.
We can notice that $\mathbb{P}( \mathcal{B}) \underset{n \rightarrow\infty}{\longrightarrow}1.$
\end{lem}
Thus if $$r_{n}\geqslant AKL \left\lbrace  C_{L,B}\vee\underset{\left\lbrace \vert x\vert\leqslant L\kappa_{n}\right\rbrace \cap\Theta}{\textrm{max}}\vert \psi^{'}(x)\vert)\right\rbrace \sqrt{\dfrac{2\log(2G_{n})}{n}}$$ with $K$ choosen such that 
$$r_{n}\geqslant \textrm{max}\left(C_{1},C_{2},C_{3}\right) $$
where $$C_{1}:= 8\sqrt{2}ALC_{L,B}\sqrt{\dfrac{\log(2G_{n})}{n}}$$
$$C_{2}:= 16A^{2}LC_{L,B}\dfrac{\log(2G_{n})}{n}$$
and $$C_{3}:=20AL\left( \underset{\left\lbrace \vert x\vert\leqslant L\kappa_{n}\right\rbrace \cap\Theta}{\textrm{max}}\vert \psi^{'}(x)\vert\right)\sqrt{\dfrac{2\log 2G_{n}}{n}}$$ then $\mathbb{P}(\mathcal{A}\cap \mathcal{B})\geqslant 1-(2d_{\textrm{max}}+C)(2G_{n})^{-A^{2}/2}.$
\end{proof}

\subsection{\textbf{Proofs of the technical lemmata}}

\textbf{Proof of Lemma~\ref{lem-moment}}
\begin{proof}
Let $\theta \in \textrm{Int}(\Theta)$ and $Y_{\theta}$ be a real random variable with density $\exp(\theta y -\psi(\theta))F(dy)$. First we prove that for all $k \in \mathbb{N}$, there exists a constant $C_{\theta}>0$ which depends on $\theta$ such that 
$$ \mathbb{E}\vert Y_{\theta}\vert^{k}\leqslant k!C_{\theta}^{k}.$$
The $k^{th}$ absolute moment of $Y_{\theta}$ is the $k^{th}$ derivate of $H_{\theta}:=\mathbb{E}(e^{s\vert Y_{\theta}\vert})$ at 0. Let $s \in \mathbb{C}$ be given. We have $$H_{\theta}(s)=\dfrac{M_{+}(s+\theta)+M_{-}(\theta-s)}{M(\theta)}$$
where we define $M_{+}$ and $M_{-}$ as  
\begin{displaymath}
M_{+}:
\left|
  \begin{array}{rcl}
     \mathbb{C} & \longrightarrow & \mathbb{C} \\
    z & \longmapsto & \int_{y\geqslant0} e^{yz}F(dy)
  \end{array}
\right.
\end{displaymath}
and

\begin{displaymath}
M_{-}:
\left|
  \begin{array}{rcl}
     \mathbb{C} & \longrightarrow & \mathbb{C} \\
    z & \longmapsto & \int_{y<0} e^{yz}F(dy).
  \end{array}
\right.
\end{displaymath}

$M$ is analytic on $\Omega_{\Theta}:=\left\lbrace z \in \mathbb{C}: \textrm{Re}(z) \in \textrm{Int}(\Theta) \right\rbrace$ and so does $M_{+}$ and $M_{-}$. Therefore $H_{\theta}:s \longmapsto \mathbb{E}(e^{s\vert Y_{\theta}\vert})$ is analytic on $U_{\theta}:=\left\lbrace s \in \mathbb{C}: \textrm{Re}(s+\theta) \in \textrm{Int}(\Theta) \quad \textrm{and} \quad\textrm{Re}(\theta-s) \in \textrm{Int}(\Theta) \right\rbrace $.
Since $\theta \in \textrm{Int}(\Theta)$, $H_{\theta}$ is analytic at the point 0 and hence the function is also holomorphic in a neighborhood of 0.
We recall the following  result for analytic functions (see \cite{WAG03}).\\
\textbf{Theorem:}
\textit{If $f$ is holomorphic on a domain $\Omega$ of $\mathbb{C}$ then $f \in \mathcal{C}^{\infty}(\Omega)$ and if in addition $\Omega$ is simply connected then for all contour $\gamma$ around $z \in \Omega$ we have
$$f^{(n)}(z)=\dfrac{n!}{2i\pi}\int_{\gamma}\dfrac{f(v)}{(v-z)^{n+1}}dv.$$}

Using the previous theorem on $H_{\theta}$ at $0$ and taking $\gamma$ as a circle with radius  $R$ centered in 0 (such that $H_{\theta}$ is holomorphic on $D(0,R)$, of course $R$ depends on $\theta$) we obtain that for all $k \in \mathbb{N}$
\begin{equation}
\vert H_{\theta}^{(k)}(0) \vert \leqslant \dfrac{k!}{2\pi}\left\vert \int_{\gamma}\dfrac{H_{\theta}(v)}{v^{k+1}}dv \right\vert\leqslant \dfrac{k!}{R^{k}}\underset{\vert z\vert\leqslant R }{\textrm{sup}}\vert H_{\theta}(z)\vert.
\label{moment}
\end{equation}
and the result follows with $C_{\theta}:=\textrm{max}\left( 1,\dfrac{1}{R}\underset{\vert z\vert\leqslant R }{\textrm{sup}}\vert H_{\theta}(z)\vert\right) $.
Thanks to assumption (H.2) we can apply~(\ref{moment}) with 
$\theta ={\beta^{*}}^{T}X$ and find that for all $k\in \mathbb{N}$,
$$\mathbb{E}\left(\vert Y \vert^{k}\vert X\right)\leqslant k!\left( C_{{\beta^{*}}^{T}X}\right) ^{k}. $$
Finally (H.1) combined with (H.3) leads to 
$$\mathbb{E}\left(\vert Y\vert^{k} \right)\leqslant k!\left(\underset{\vert \theta\vert\leqslant LB}{\textrm{sup}}C_{\theta} \right) ^{k}$$
and the result follows with $C_{L,B}:=\underset{\vert \theta\vert\leqslant LB}{\textrm{sup}}C_{\theta}$.
\end{proof}

\textbf{Proof of Lemma~\ref{local}}
\begin{proof}
The proof is based on the convexity of the loss function and of the penalty, the main idea of the proof is similar to the one used by B\"uhlmann and van de Geer \cite{BUH11} for the Lasso to show consistency of the excess of risk. Define $t:= \dfrac{M}{M+\sum_{g=1}^{G_{n}}\sqrt{d_{g}}\Vert \hat{\beta}_n^{g}-{\beta^{*}}^g\Vert_{2}}$ and $\tilde{\beta}:=t\hat{\beta}_n+(1-t)\beta^{*}$. Notice $\sum_{g=1}^{G_{n}}\sqrt{d_{g}}\Vert \tilde{\beta}^{g}-{\beta^{*}}^g\Vert_{2}\leqslant M$.
 By convexity of $\beta \rightarrow l_{\psi}(\beta)$ and $\beta \rightarrow \Vert \beta\Vert_{2}$ combined with the fact that $\hat{\beta}_n$ satisifies (\ref{eq debut-gpl}) we find 
 $$\mathbb{P}\left( l(\tilde{\beta})-l(\beta^{*})\right)+2r_{n}\sum_{g=1}^{G_{n}}\sqrt{d_{g}}\Vert \tilde{\beta}^{g}\Vert_{2}$$
 $$\leqslant (\mathbb{P}_{n}-\mathbb{P})\left( l(\beta^{*})-l(\tilde{\beta})\right) +2r_{n}\sum_{g=1}^{G_{n}}\sqrt{d_{g}}\Vert {\beta^{*}}^g\Vert_{2} .$$
On the event $\mathcal{A}\bigcap \mathcal{B}$ we have 
$$\mathbb{P}\left( l(\tilde{\beta})-l(\beta^{*})\right)+2r_{n}\sum_{g=1}^{G_{n}}\sqrt{d_{g}}\Vert \tilde{\beta}^{g}\Vert_{2}$$
$$\leqslant r_{n}\sum_{g=1}^{G_{n}}\sqrt{d_{g}}\Vert \tilde{\beta}^{g}-{\beta^{*}}^{g}\Vert_{2}+r_{n}\dfrac{\varepsilon_{n}}{2} +2r_{n}\sum_{g=1}^{G_{n}}\sqrt{d_{g}}\Vert {\beta^{*}}^g\Vert_{2}.$$ 
Because $\mathbb{P}\left( l(\tilde{\beta})-l(\beta^{*})\right)\geqslant 0$, by adding to both sides of the inequality $2r_{n}\sum_{g=1}^{G_{n}}\sqrt{d_{g}}\Vert {\beta^{*}}^g\Vert_{2} $ and by using the triangular inequality we have
$$\sum_{g=1}^{G_{n}}\sqrt{d_{g}}\Vert \tilde{\beta}^{g}-{\beta^{*}}^{g}\Vert_{2}\leqslant \dfrac{\varepsilon_{n}}{2} +4\sum_{g=1}^{G_{n}}\sqrt{d_{g}}\Vert {\beta^{*}}^g\Vert_{2}.$$
Therefore, using (H.3), we have $$\sum_{g=1}^{G_{n}}\sqrt{d_{g}}\Vert \tilde{\beta}^{g}-{\beta^{*}}^{g}\Vert_{2}\leqslant \dfrac{\varepsilon_{n}}{2} +4B=\dfrac{M}{2}.$$ 
i.e. $$t\sum_{g=1}^{G_{n}}\sqrt{d_{g}}\Vert \hat{\beta}_n^{g}-{\beta^{*}}^{g}\Vert_{2}\leqslant \dfrac{M}{2}$$ and then the definition of $t$ leads to
$$\sum_{g=1}^{G_{n}}\sqrt{d_{g}}\Vert \hat{\beta}_n^{g}-{\beta^{*}}^{g}\Vert_{2}\leqslant M.$$
\end{proof}

\textbf{Proof of Lemma~\ref{lem:propA}}
\begin{proof}
We have $$\mathbb{P}(\mathcal{A}^{c})\leqslant \sum_{g=1}^{G_{n}}\mathbb{P}\left\lbrace \left\Vert \frac{1}{n}\sum_{i=1}^{n}\left( Y_{i}X_{i}^{g}-\mathbb{E}(YX^{g})\right) \right\Vert_{2}^{2}> \dfrac{r_{n}^{2}}{4}d_{g} \right\rbrace $$
\begin{equation}\label{eq:majo-proba-A}
\leqslant  \sum_{g=1}^{G_{n}}\sum_{j=1}^{d_{g}}\mathbb{P}\left\lbrace  \dfrac{1}{n}\left\vert\sum_{i=1}^{n}\left( Y_{i}X_{i,j}^{g}-\mathbb{E}\left( YX_{j}^{g}\right) \right) \right\vert >\dfrac{r_{n}}{2}\right\rbrace. 
\end{equation}
For $j=1,...,d_{g}$ and $i=1,...,n$, let 
$$W_{ij}^{g}:= Y_{i}X_{i,j}^{g}-\mathbb{E}\left( Y_{i}X_{j}^{g}\right).$$
The random variables $\left\lbrace W_{ij}\right\rbrace _{i=1,...,n}$ are independant, identically distributed and centered and for all $m\geqslant 2$	$$\mathbb{E}\vert W_{ij}^{g} \vert ^{m}\leqslant\sum_{k=0}^{m}\binom{k}{m}\mathbb{E}\vert Y_{i}X_{i,j}\vert^{k}\left(  \mathbb{E} \vert Y_{i}X_{i,j}\vert\right) ^{m-k} $$
By using Jensen inequality we obtain
$$\mathbb{E}\vert W_{ij}^{g} \vert^{m}\leqslant 2^{m}\underset{k=1,...,m}{\textrm{max}}\left\lbrace \mathbb{E}\vert Y_{i}X_{i,j}\vert^{k}\mathbb{E}\vert Y_{i}X_{i,j}\vert^{m-k} \right\rbrace. $$
For all $k \in \mathbb{N} $, by (H.1) and Lemma~\ref{lem-moment} we have $$\mathbb{E}\vert Y_{i}X_{i,j} \vert^{k}\leqslant L^{k}k!\left( C_{L,B}\right) ^{k}. $$
Therefore $\mathbb{E}\vert W_{ij}^{g}\vert^{m} \leqslant m!(2LC_{L,B})^{m} $.
Hence the conditions are satisfied to apply Bernstein concentration inequality \cite{BEN62} with $K=2LC_{L,B}$ and $\sigma^{2}=8(LC_{L,B})^{2}$. Thus we obtain
$$\mathbb{P}\left( \dfrac{1}{n} \left\lvert \sum_{i=1}^{n} W_{ij}^{g} \right\rvert> r_{n}/2\right)$$
\begin{equation}\label{eq:majo-proba-A-bis}
\leqslant  2\left(\exp\left(\dfrac{-nr_{n}}{16LC_{L,B}}\right)+ \exp\left(\dfrac{-nr_{n}^{2}}{32(2LC_{L,B})^{2}}\right) \right).
\end{equation}
Finally, from (\ref{eq:majo-proba-A}) and (\ref{eq:majo-proba-A-bis}), we deduce that $\mathbb{P}(\mathcal{A}^{^{c}})$ is bounded by 
$$2d_{\textrm{max}}G_{n}\left(\exp\left(\dfrac{-
nr_{n}}{16LC_{L,B}}\right)+ \exp\left(\dfrac{-nr_{n}^{2}}{32(LC_{L,B})^{2}}\right) \right).$$ 
Therefore if 
$$r_{n}\geqslant A^{2}16LC_{L,B}\dfrac{\log(2G_{n})}{n}\vee A8\sqrt{2}LC_{L,B}\sqrt{\dfrac{\log(2G_{n})}{n}}$$ with $A>1$ then 
$$\mathbb{P}\left\lbrace \mathcal{A}^{C}\right\rbrace\leqslant 2d_{\textrm{max}}(2G_{n})^{1-A^{2}}.$$
\end{proof}

\textbf{Proof of Lemma~\ref{lem:propB}}
\begin{proof} 
The proof rests on the following Lemma
\begin{lem}\label{lem:upbound2}
Let $R>0$ be given. Define $$Z_{R}:=\underset{\sum_{g=1}^{G_{n}}\sqrt{d_{g}}\Vert \beta^{g}-{\beta^{*}}^{g}\Vert_{2}\leqslant R}{\sup}\left\lbrace \left\lvert(\mathbb{P}_{n}-\mathbb{P})\left( l_{\psi}(\beta^{*})-l_{\psi}(\beta)\right)\right\rvert\right\rbrace. $$
If $A\geqslant 1$ then
\begin{equation}\mathbb{P}\left( Z_{R}\geqslant A5DLR\sqrt{\dfrac{2\log 2G_{n}}{n}}\right) \leqslant(2G_{n})^{-A^{2}}
\label{proba majo2-gpl}
\end{equation}
where $D:= \underset{\left\lbrace \vert x\vert\leqslant L(R+B)\right\rbrace \cap\Theta}{\textrm{max}}\left\lbrace \vert\psi^{'}(x)\vert \right\rbrace $.
\end{lem}

\begin{proof}
Let $\beta$ satisfy $\sum_{g=1}^{G_{n}}\sqrt{d_{g}}\Vert \beta^{g}-{\beta^{*}}^{g}\Vert_{2}\leqslant R$. Notice that if we change  $X_{i}$ by $X_{i}^{'}$ while keeping the others fixed then $Z_{R}$ is modified of at most $\dfrac{2}{n}LR\exp\left(  L(R+B)\right)$.
To see this let $$\mathbb{P}_{n}=\dfrac{1}{n}\sum_{j=1}^{n}1_{X_{j},Y_{j}}$$ and  $$\mathbb{P}_{n}^{'}=\dfrac{1}{n}\sum_{j=1, j\neq i}^{n}1_{X_{j},Y_{j}}+1_{X_{i}^{'},Y_{i}^{'}}$$ then we have
$$(\mathbb{P}_{n}-\mathbb{P})(l_{\psi}(\beta^{*})-l_{\psi}(\beta))-(\mathbb{P}_{n}^{'}-\mathbb{P})(l_{\psi}(\beta^{*})-l_{\psi}(\beta))$$
$$=\frac{1}{n}\left(l_{\psi} (\beta^{*},X_{i})-l_{\psi} (\beta,X_{i})-l_{\psi} (\beta^{*},X_{i}^{'})+l_{\psi} (\beta,X_{i}^{'})\right) $$ 
$$\leqslant \frac{1}{n}\vert\psi^{'}(\tilde{\beta}^{T}X_{i})\vert\vert {\beta^{*}}^{T}X_{i}-\beta^{T}X_{i}\vert+ \frac{1}{n}\vert\psi^{'}(\tilde{\beta}^{T}X_{i}^{'})\vert\vert {\beta^{*}}^{T}X_{i}^{'}-\beta^{T}X_{i}^{'}\vert $$
with $\tilde{\beta}^{T}X_{i}$ which is an intermediate point between $\beta^{T}X_{i}$ and ${\beta^{*}}^{T}X_{i}$ (using a first order Taylor expansion of the exponential function). Then, using the same argument as for~(\ref{eq:majo-taylor}), we have
$$\vert \tilde{\beta}^{T}X_{i}\vert\leqslant LR+LB .$$
Therefore 
$$(\mathbb{P}_{n}-\mathbb{P})(l_{\psi}(\beta^{*})-l_{\psi}(\beta))-(\mathbb{P}_{n}^{'}-\mathbb{P})(l_{\psi}(\beta^{*})-l_{\psi}(\beta))$$
$$\leqslant \frac{2}{n}LR \underset{\left\lbrace \vert x\vert\leqslant L(R+B)\right\rbrace \cap\Theta}{\textrm{max}}\vert\psi^{'}(x)\vert=\frac{2}{n}LRD. $$
We can apply McDiarmid's inequality also called the bounded difference inequality.\\
\textbf{Theorem.}
\textit{Let $A$ a set. Assume $g:A^{N}\rightarrow\mathbb{R}$ is a function that satisfies the bounded difference inequality
$$\underset{x_{1},...,x_{n},x_{i}^{'}\in A}{\textrm{sup}}\vert g(x_{1},...,x_{n})-g(x_{1},...,x_{i-1},x_{i}^{'},x_{i+1},...,x_{n})\vert\leqslant c_{i}.$$
Let $X_{1},..,X_{n}$ be independent random variables all taking values in the set $A$. Then for all $t>0$,
$$\mathbb{P}\left\lbrace g(X_{1},...,X_{n})-\mathbb{E}g(X1,...,X_{n})\geqslant t \right\rbrace \leqslant e^{-2t^{2}/\sum_{i=1}^{n}c_{i}^{2}}.$$}
We can apply McDiarmid's inequality to $Z_{R}$ and obtain 
$$\mathbb{P}(Z_{R}-\mathbb{E}Z_{R}\geqslant u)\leqslant \exp\left( -\frac{nu^{2}}{2R^{2}L^{2}\left(D \right)^{2}  }\right) .$$
Therefore if $r_{n}\geqslant AD LR\sqrt{\dfrac{2\log 2G_{n}}{n}}$ with $A>0$ then 
\begin{equation}
\mathbb{P}(Z_{R}-\mathbb{E}Z_{R}\geqslant r_{n})\leqslant (2G_{n})^{-A^{2}}.
\label{proba1-gpl}
\end{equation}
Now we have to bound the mean $\mathbb{E}Z_{R}$.
\begin{lem}
$$\mathbb{E}{Z_{R}}\leqslant 4RLD \sqrt{\dfrac{2\log(2G_{n})}{n}}.$$
\end{lem}

\begin{proof}
First let us introduce two theorems. Let  $X_{1},...,X_{n}$ independent random variables with values in some space $\mathcal{X}$ and $\mathcal{F}$ a class of real-valued functions on $\mathcal{X}$.\\
\textbf{Theorem: Symmetrization theorem~\cite{WELL96}.}
\textit{ Let $\epsilon_{1},...,\epsilon_{n}$ be Rademacher sequence independent of $X_{1},...,X_{n}$ and $f\in \mathcal{F}$. Then
$$\mathbb{E}\left( \underset{f \in \mathcal{F}}{\sup}\left\lvert \sum_{i=1}^{n}\left\lbrace f(X_{i})-\mathbb{E}(f(X_{i})) \right\rbrace \right\rvert\right)$$
$$ \leqslant 2\mathbb{E}\left( \underset{f \in \mathcal{F}}{\sup}\vert \sum_{i=1}^{n} \epsilon_{i}f(X_{i}) \vert\right). $$}
\textbf{Theorem: Contraction principle~\cite{LEDOUX91}.} \textit{Let $x_{1},...,x_{n}$ elements of $\mathcal{X}$ and $\varepsilon_{1},...,\varepsilon_{n}$ be Rademacher sequence. Consider Lipschitz functions $g_{i}$. Then for any function $f$ and $h$ in $\mathcal{F}$, we have 
$$\mathbb{E}\left( \underset{f \in \mathcal{F}}{\sup}\left\lvert \sum_{i=1}^{n}\epsilon_{i}\left\lbrace g_{i}(f(x_{i}))- g_{i}(h(x_{i}))\right\rbrace \right\rvert\right)$$
$$ \leqslant 2\mathbb{E}\left(\underset{f\in \mathcal{F}}{\sup}\left\lvert \sum_{i=1}^{n}\epsilon_{i} (f(x_{i})- h(x_{i}))\right\rvert \right) $$
}
Let $\epsilon_{1},...,\epsilon_{n}$ a Rademacher sequence independant of $X_{1},...,X_{n}$ and let $\mathcal{S}_{R}:=\left\lbrace \beta \in \mathbb{R}^{p}: \sum_{g=1}^{G_{n}}\sqrt{d_{g}}\Vert \beta^{g}-{\beta^{*}}^g\Vert_{2}\leqslant  R \right\rbrace $. Then by the Symmetrization theorem and the Contraction theorem  ($\psi$ is $D$-lipschitz on the compact set $\mathcal{S}_{R}$ ) we have
$$\mathbb{E}{Z_{R}}\leqslant 4D\mathbb{E}\left(\underset{\beta \in \mathcal{S}_{R}}{\sup}\frac{1}{n}\sum_{i=1}^{n}\left\lvert \epsilon_{i}({\beta^{*}}^{T}X_{i}-\beta^{T}X_{i}) \right\rvert \right) $$
$$\leqslant4DR\mathbb{E}\left(\underset{g \in  \left\lbrace 1,...,G_{n} \right\rbrace }{\max}\left\lvert \frac{1}{n}\sum_{i=1}^{n}\epsilon_{i}\dfrac{\Vert X_{i}^{g}\Vert_{2}}{\sqrt{d_{g}}}\right\rvert \right), $$
by Holder inequality for the last bound.
Now we are going to use the following theorem which is a consequence of Hoeffding inequality.\\
\textbf{Theorem.}
\textit{Let $X_{1},...,X_{n}$ be independent random variables on  $\mathcal{X}$ and $f_{1},...,f_{n}$ real-valued functions on $\mathcal{X}$ which satisfies for all $j=1,...,p$ and all $i=1,...,n$
$$\mathbb{E}f_{j}(X_{i})=0, \quad \vert f_{j}(X_{i})\vert \leqslant a_{ij}.$$
Then $$\mathbb{E}\left( \underset{1\leqslant j\leqslant p}{\max}\left\lvert \sum_{i=1}^{n}f_{j}(X_{i})\right\rvert\right) \leqslant\sqrt{2\log(2p)}\underset{1\leqslant j\leqslant p}{\max}\sqrt{\sum_{i=1}^{n}a_{ij}^{2}}.$$}
By applying this theorem we obtain
$$\mathbb{E}\left(\underset{g \in  \left\lbrace 1,...,G_{n} \right\rbrace }{\max}\left\lvert \frac{1}{n}\sum_{i=1}^{n}\epsilon_{i}\dfrac{\Vert X_{i}^{g}\Vert_{2}}{\sqrt{d_{g}}} \right\rvert\right)\leqslant L\sqrt{\dfrac{2\log(2G_{n})}{n}}.$$
Thus
\begin{equation}
\mathbb{E}{Z_{R}}\leqslant 4RLD\sqrt{\dfrac{2\log(2G_{n})}{n}}.
\label{proba2-gpl}
\end{equation}
\end{proof}
So we can conclude from (\ref{proba1-gpl}) and (\ref{proba2-gpl}) that if $A\geqslant 1$ then
\begin{equation}\mathbb{P}\left( Z_{R}\geqslant A5DLR\sqrt{\dfrac{2\log 2G_{n}}{n}}\right) \leqslant(2G_{n})^{-A^{2}}
\label{proba majo2-gpl}
\end{equation}
for all $R>0$.
\end{proof}

Split up $$\left\lbrace \beta \in \mathbb{R}^{p}:\sum_{g=1}^{G_{n}}\sqrt{d_{g}}\Vert \beta^{g}-{\beta^{*}}^g\Vert_{2}\leqslant M \right\rbrace, $$  where $M=8B+\varepsilon_{n}$, into two sets which are
$$E_{1}=\left\lbrace \beta : \sum_{g=1}^{G_{n}}\sqrt{d_{g}}\Vert \beta^{g}-{\beta^{*}}^g\Vert_{2}\leqslant \varepsilon_{n} \right\rbrace, $$
$$E_{2}=\left\lbrace  \beta: \varepsilon_{n}\leqslant \sum_{g=1}^{G_{n}}\sqrt{d_{g}}\Vert \beta^{g}-{\beta^{*}}^g\Vert_{2}\leqslant  M \right\rbrace$$
$$ \subseteq \bigcup_{j=1}^{j_{n}}\left\lbrace  \beta: 2^{j-1}\varepsilon_{n} <\sum_{g=1}^{G_{n}}\sqrt{d_{g}}\Vert \beta^{g}-{\beta^{*}}^g\Vert_{2}\leqslant 2^{j}\varepsilon_{n} \right\rbrace$$
where  $j_{n}:= \lfloor \log_{2}(nM)\rfloor+1$ is the smaller integer such that $2^{j_{n}}\varepsilon_{n}\geqslant M$. We recall that 
$$\nu_{n}(\beta,\beta^{*}):=\dfrac{(\mathbb{P}_{n}-\mathbb{P})\left( l_{\psi}(\beta^{*})-l_{\psi}(\beta)\right)}{\sum_{g=1}^{G_{n}}\sqrt{d_{g}}\Vert \beta^{g}-{\beta^{*}}^g\Vert_{2} +\varepsilon_{n}}$$ 
and to simplify notation let
$$\alpha_{n}(\beta, \beta^{*}):=(\mathbb{P}_{n}-\mathbb{P})\left( l_{\psi}(\beta^{*})-l_{\psi}(\beta)\right)$$ and $$\Phi(t):= \underset{\left\lbrace \vert x\vert\leqslant t\right\rbrace \cap\Theta}{\textrm{max}}\vert \psi^{'}(x)\vert.$$ Let $A\geqslant1$. We recall that $\kappa_{n}:=17B+\dfrac{2}{n}=2M+B$. On the event $E_{1}$, 
$$\mathbb{P}\left(\underset{\beta \in E_{1}}{\sup}\left \lvert  \nu_{n}(\beta,\beta^{*}) \right\rvert  
\geqslant A10L\Phi(L\kappa_{n})\sqrt{\dfrac{ 2\log(2G_{n})}{n}}\right)$$
$$\leqslant \mathbb{P}\left(\underset{\beta \in E_{1}}{\sup}\left \lvert\alpha_{n}(\beta, \beta^{*}) \right\rvert  \geqslant A10L\Phi(L\kappa_{n})\varepsilon_{n}\sqrt{\dfrac{ 2\log(2G_{n})}{n}}\right)$$
$$\leqslant \mathbb{P}\left(\underset{\beta \in E_{1}}{\sup}\left \lvert \alpha_{n}(\beta, \beta^{*}) \right\rvert  \geqslant A5L\Phi(L(\varepsilon_{n}+B))\varepsilon_{n}\sqrt{\dfrac{ 2\log(2G_{n})}{n}}\right)$$
given that $2M\geqslant \varepsilon_{n}$. From Lemma~\ref{lem:upbound2} with $R=\varepsilon_{n}$ we deduce
$$\mathbb{P}\left(\underset{\beta \in E_{1}}{\sup}\left \lvert \nu_{n}(\beta,\beta^{*}) \right\rvert  
\geqslant A10L\Phi(L\kappa_{n})\sqrt{\dfrac{ 2\log(2G_{n})}{n}}\right)$$
\begin{equation}
\leqslant (2G_{n})^{-A^{2}}.
\label{E2}
\end{equation}
On the event $E_{2}$, using the same type of argument as for (\ref{E2}) with  $R=2^{j}\varepsilon_{n}$ (given that $2M\geqslant 2^{j}\varepsilon_{n}$) for all $j=1,...,j_{n}$ , we find
$$\mathbb{P}\left(\underset{\beta \in E_{2}}{\sup}\left \lvert \nu_{n}(\beta,\beta^{*}) \right\rvert  \geqslant A10L\Phi(L\kappa_{n})\sqrt{\dfrac{ 2\log(2G_{n})}{n}}\right)$$
$$\leqslant j_{n}(2G_{n})^{-A^{2}}.$$
Finally we have
\begin{equation}
\leqslant C^{'} (2G_{n})^{-\frac{A^{2}}{2}}
\label{E1}
\end{equation}where $C^{'}$ is a constant (because $j_{n}=\lfloor \log_{2}(nM)\rfloor+1$ and $n\ll G_{n}$) and the result of Lemma~\ref{lem:propB} follows from (\ref{E2}) and (\ref{E1}) with $C=1+C^{'}$.
\end{proof}

\section{\textbf{Proof of Theorem \ref{theo-elastic}}}

The main step of the proof are the same as for the Lasso.
\begin{proof}
By the same arguments as the ones used to prove \eqref{eq-principale-gpl} we have
$$ \mathbb{P}\left( l(\beta^{*})-l(\hat{\beta}_n)\right)+2r_{n}\Vert \hat{\beta}_n\Vert_{1}+ t_{n}\Vert \hat{\beta}_n\Vert_{2}^{2}$$
\begin{equation}
\leqslant \left( \mathbb{P}_{n}-\mathbb{P}\right) \left( l(\beta^{*})-l(\hat{\beta}_n)\right) +2r_{n} \Vert\beta\Vert_{1}+t_{n}\Vert \beta^{*}\Vert_{2}^{2}.
\label{eq-principale-elastic}
\end{equation}
The upper bound of $(\mathbb{P}_{n}-\mathbb{P})\left( l_{l}(\beta^{*})-l_{l}(\hat{\beta}_n)\right)$, of $(\mathbb{P}_{n}-\mathbb{P})\left( l_{\psi}(\beta^{*})-l_{\psi}(\hat{\beta}_n)\right)$ and the lower bound of $\mathbb{P}\left( l(\beta^{*})-l(\hat{\beta}_n)\right)$ remains the same as those presented in the proof of Theorem~\ref{th:LassoGLM} (see Proposition \ref{prop:upbound1}, Proposition \ref{prop:upbound2} and Proposition \ref{prop:lowbound}). Once these three propositions are proved, the rest of the proof is similar to the one for logistic regression presented in \cite{BUN08}.
On the event $\mathcal{A}\cap \mathcal{B}$ (which occurs with probability at least $1-2(2p)^{1-A^{2}}-C^{'}(2p)^{-A^{2}/2}\geqslant 1-C(2p)^{-A^{2}/2}$) by adding $r_{n}\Vert \hat{\beta}_n-\beta^{*}\Vert_{1}$ and $t_{n}\sum_{j\in I^{*}}(\beta_{j}^{*}-{\hat{\beta}}_{j})^{2}$ to both sides of the inequality (\ref{eq-principale-elastic}), we have
$$r_{n}\Vert \hat{\beta}_n-\beta^{*}\Vert_{1}+\mathbb{P}\left( l(\beta^{*})-l(\hat{\beta}_n)\right)+t_{n}\sum_{j\in I^{*}}(\beta_{j}^{*}-\hat{\beta}_{j})^{2}$$
$$\leqslant  2r_{n}\Vert \hat{\beta}_n-\beta^{*}\Vert_{1}+ 2r_{n}\Vert \beta^{*}\Vert_{1}-2r_{n}\Vert \hat{\beta}_n\Vert_{1}+t_{n}\sum_{j\in I^{*}}(\beta_{j}^{*}-\hat{\beta}_{j})^{2}$$
\begin{equation}
-t_{n}\Vert \hat{\beta}_n\Vert_{2}^{2}+t_{n}\Vert \beta^{*}\Vert_{2}^{2}+\frac{r_{n}}{2}\varepsilon_{n}
\label{eq-elnet}
\end{equation}
with $\varepsilon_{n}=\dfrac{1}{n}$. 
On one hand, by the same argument as for (\ref{eq-principale2-gpl}) we get $$2r_{n}\Vert \hat{\beta}_n-\beta^{*}\Vert_{1}+ 2r_{n}\Vert \beta^{*}\Vert_{1}-2r_{n}\Vert \hat{\beta}_n\Vert_{1}\leqslant 4r_{n}\sum_{j\in I^{*}}\vert \beta_{j}^{*}-\hat{\beta}_{j}\vert $$ 
and on the other hand
 $$t_{n}\sum_{j\in I^{*}}(\beta_{j}^{*}-\hat{\beta}_{j})^{2}-t_{n}\Vert \hat{\beta}_n\Vert_{2}^{2}+t_{n}\Vert \beta^{*}\Vert_{2}^{2}$$
 $$\leqslant 2t_{n}\sum_{j \in I^{*}}{\beta_{j}^{*}}^{2}-2t_{n}\sum_{j \in I^{*}}\beta_{j}^{*}\hat{\beta}_{j}$$
 $$\leqslant r_{n}\sum_{j \in I^{*}}\vert \beta_{j}^{*}-\hat{\beta}_{j}\vert.$$
Therefore inequality (\ref{eq-elnet}) can be bounded by
$$r_{n}\Vert \hat{\beta}_{n}-\beta^{*}\Vert_{1}+\mathbb{P}\left( l(\hat{\beta}_{n})-l(\beta^{*})\right)+t_{n}\sum_{j\in I^{*}}(\beta_{j}^{*}-\hat{\beta}_{j})^{2}$$
$$\leqslant 4r_{n}\sum_{j\in I^{*}}\vert \beta_{j}^{*}-\hat{\beta}_{j}\vert+r_{n}\sum_{j \in I^{*}}\vert \beta_{j}^{*}-\hat{\beta}_{j}\vert+\frac{r_{n}}{2}\varepsilon_{n}.$$
Since $\mathbb{P}\left( l(\beta^{*})-l(\hat{\beta}_{n})\right)\geqslant0$ we have $$r_{n}\Vert \hat{\beta}_{n}-\beta^{*}\Vert_{1}\leqslant 5r_{n}\sum_{j \in I^{*}}\vert \beta_{j}^{*}-\hat{\beta}_{j}\vert +\frac{r_{n}}{2}\varepsilon_{n}$$
and then $$\sum_{j \in {I^{*}}^{C}}\vert \beta_{j}^{*}-\hat{\beta}_{j}\vert\leqslant4\sum_{j \in I^{*}}\vert \beta_{j}^{*}-\hat{\beta}_{j}\vert+\frac{\varepsilon_{n}}{2}.$$
Thus ($\hat{\beta}_{n}-\beta^{*}) \in S(4,\dfrac{\varepsilon_{n}}{2})$.
Using Proposition \ref{prop:lowbound} in the case of groups of size one we find
$$r_{n}\Vert \hat{\beta}_{n}-\beta^{*}\Vert_{1}+t_{n}\sum_{j\in I^{*}}(\beta_{j}^{*}-\hat{\beta}_{j})^{2}+c_{n}\mathbb{E}\left( \hat{\beta}_{n}^{T}X-{\beta^{*}}^{T}X\right) ^{2}$$
$$\leqslant  5r_{n}\sum_{j \in I^{*}}\vert \beta_{j}^{*}-\hat{\beta}_{j}\vert+\frac{r_{n}}{2}\varepsilon_{n},$$
with $c_{n}:=\underset{\left\lbrace\vert x\vert\leqslant L(9B+\frac{1}{n})\right\rbrace \cap\Theta}{\textrm{min}}\left\lbrace  \frac{\psi^{''}(x)}{2}\right\rbrace $.
The rest of the proof follows the guidelines of the proof of (\ref{eq:esti-error}) and  (\ref{eq:predict-error}) and leads to 
$$\Vert \hat{\beta}_{n}-\beta^{*}\Vert_{1}\leqslant \dfrac{(2.5)^{2}r_{n} s^{*}}{t_{n}+c_{n}k}+(1+\frac{1}{r_{n}})\dfrac{\varepsilon_{n}}{2}.$$
and $$\mathbb{E}\left( \hat{\beta}_{n}^{T}X-{\beta^{*}}^{T}X\right) ^{2}\leqslant  \dfrac{2(2.5)^{2}}{c_{n}k(t_{n}+c_{n}k)}r_{n}^{2}s^{*}+\dfrac{2r_{n}+3}{2nc_{n}}.$$
\end{proof}

\bibliographystyle{plain}
\bibliography{biblio-glm-IEEE}

\end{document}